\documentclass[a4paper,12pt]{amsart}
\usepackage[margin=3cm]{geometry}
\usepackage{url}
\usepackage[english]{babel} 

\usepackage{mathtools}
\usepackage{amsmath}
\usepackage{amssymb}
\usepackage{mathrsfs}  
\usepackage{bm}
\usepackage{algorithm}
\usepackage[utf8]{inputenc}
\usepackage[english,colorinlistoftodos,textsize=tiny]{todonotes}
\usepackage{comment} 
\usepackage{nicefrac}
\usepackage{pgfplots}
\pgfplotsset{compat=newest}

\usepgfplotslibrary{groupplots}
\usepgfplotslibrary{dateplot}

\usepackage{tikz}
\usepackage{pgfplots}
\usepackage{pgfplotstable}

\usepackage[noexternal]{tikzconditional}
\pgfplotsset{compat=1.7} 
\usepackage{booktabs}

\usepackage{hyperref}

  \usepackage{amsthm}
  \theoremstyle{plain}             
  \newtheorem{theorem}{Theorem}[section]
  \newtheorem{assumption}[theorem]{Assumption}
  \newtheorem{remark}[theorem]{Remark}
  \newtheorem{lemma}[theorem]{Lemma}
  \newtheorem{proposition}[theorem]{Proposition}
  
  \newtheorem{definition}[theorem]{Definition}
  \newtheorem{example}[theorem]{Example}
  \usepackage{algorithm}

\usepackage{algorithmic}

\newcommand{\bs}[1]{\boldsymbol{#1}}

\renewcommand*{\Gamma}{\varGamma}
\renewcommand*{\Delta}{\varDelta}
\renewcommand*{\Theta}{\varTheta}
\renewcommand*{\Lambda}{\varLambda}
\renewcommand*{\Xi}{\varXi}
\renewcommand*{\Pi}{\varPi}
\renewcommand*{\Sigma}{\varSigma}
\renewcommand*{\Upsilon}{\varUpsilon}
\renewcommand*{\Phi}{\varPhi}
\renewcommand*{\Psi}{\varPsi}
\renewcommand*{\Omega}{\varOmega}

\newcommand*{\latinabbr}[1]{#1}
\newcommand*{\ie}{\latinabbr{i.e.}\;} 
\newcommand*{\eg}{\latinabbr{e.g.}\;} 

\DeclareMathOperator{\ddiv}{div}
\DeclareMathOperator{\argmin}{argmin}

\DeclareMathOperator{\Span}{span}
\DeclareMathOperator{\err}{err}

\newcommand*{\abs}[2][{}]{\ensuremath{#1\lvert#2#1\rvert}}
\newcommand*{\norm}[2][{}]{\ensuremath{#1\lVert#2#1\rVert}}

\def\beq{\begin{equation}}
\def\eeq{\end{equation}}

\def\mcB{{\mathcal{B}}}

\def\mcG{{\mathcal{G}}}

\def\mcN{{\mathcal{N}}}

\def\mcO{{\mathcal{O}}}

\def\mcU{{\mathcal{U}}}
\def\mcV{{\mathcal{V}}}

\def\mcX{{\mathcal{X}}}
\def\mcY{{\mathcal{Y}}}

\def\bbR{\mathbb{R}}
\def\bbN{\mathbb{N}}
\def\bbP{\mathbb{P}}

\DeclareMathOperator{\ComputeTrafo}{ComputeTransport}
\DeclareMathOperator{\GenerateONB}{GenerateONB}
\DeclareMathOperator{\GenerateSamples}{GenerateSamples}
\DeclareMathOperator{\ReconstructTT}{ReconstructTT}

\newcommand{\on}[1]{{\operatorname{#1}}}
\newcommand{\isdef}{\mathrel{\mathrel{\mathop:}=}}

\newlength{\myl}
\newlength{\myll}

 \newcommand{\ownTitle}{Low-rank tensor reconstruction of concentrated densities with application to Bayesian inversion}   
 \newcommand{\ownShortTitle}{Tensor reconstruction for densities}
 \newcommand{\ownAMS}{%
 62F15, 
 62G07, 
 35R60, 
 60H35, 
 65C20, 
 65N12, 
 65N22, 
 65J10
 }
 \newcommand{\ownKeywords}{Tensor train, uncertainty quantification, VMC, low-rank, reduced order model, Bayesian inversion, partial differential equations with random coefficients}
 \newcommand{\ownThanks}{M. Eigel and R. Gruhlke have been supported by the DFG SPP1886 ``Polymorphic uncertainty modelling for the numerical design of structures``.}

\begin{document}
 \title[\ownShortTitle]{\ownTitle}
 \date{\today}
 
 \author[M.\ Eigel]{Martin Eigel}
 \address{Weierstrass Institute\\Mohrenstrasse 39\\D-10117 Berlin\\Germany}
 \email{martin.eigel@wias-berlin.de}
 
 \author[R.\ Gruhlke]{Robert Gruhlke}
 \address{Weierstrass Institute\\Mohrenstrasse 39\\D-10117 Berlin\\Germany}
 \email{robert.gruhlke@wias-berlin.de}
 
 \author[M.\ Marschall]{Manuel Marschall}
 \address{Research conducted at Weierstrass Institute\\Mohrenstrasse 39\\D-10117 Berlin\\Germany. Present address: Physikalisch-Technische Bundesanstalt\\ Braunschweig and Berlin\\ Germany}
 \email{manuel.marschall@ptb.de}
 
 \subjclass[2010]{%
 \ownAMS
 }
 \thanks{\ownThanks}
 \begin{abstract}
  Transport maps have become a popular mechanic to express complicated probability densities using sample propagation through an optimized push-forward. 
  Beside their broad applicability and well-known success, transport maps suffer from several drawbacks such as numerical inaccuracies induced by the optimization process and the fact that sampling schemes have to be employed when quantities of interest, \eg moments are to compute. 
  This paper presents a novel method for the accurate functional approximation of probability density functions (PDF) that copes with those issues.
  By interpreting the pull-back result of a target PDF through an inexact transport map as a perturbed reference density, a subsequent functional representation in a more accessible format allows for efficient and more accurate computation of the desired quantities.
  We introduce a layer-based approximation of the perturbed reference density in an appropriate coordinate system to split the high-dimensional representation problem into a set of independent approximations for which separately chosen orthonormal basis functions are available. This effectively motivates the notion of h- and p-refinement (i.e. ``mesh size'' and polynomial degree) for the approximation of high-dimensional PDFs.
  To circumvent the curse of dimensionality and enable sampling-free access to certain quantities of interest, a low-rank reconstruction in the tensor train format is employed via the Variational Monte Carlo method.
  An a priori convergence analysis of the developed approach is derived in terms of Hellinger distance and the Kullback-Leibler divergence.
  Applications comprising Bayesian inverse problems and several degrees of concentrated densities illuminate the (superior) convergence in comparison to Monte Carlo and Markov-Chain Monte Carlo methods.

\end{abstract}

  \maketitle

  \section{Overview}
  We derive a novel numerical method for the functional representation of complicated (in particular highly concentrated) probability densities.
  This difficult task usually is attacked with Markov Chain Monte Carlo\\ (MCMC) methods which yield samples of the posterior.
  Despite their popularity, the convergence rate of these methods is ultimately limited by the employed Monte Carlo sampling technique, see e.g.~\cite{dodwell2019multilevel} for recent multilevel techniques in this context.
  Moreover, practical issues e.g. regarding the initial number of samples (burn-in) or a specific convergence assessment arise.
  
  In this work, we propose a new approach based on \emph{function space representations with efficient surrogate models} in several instances.
  This is motivated by our previous work on adaptive low-rank approximations of solutions of parametric random PDEs with Adaptive Stochastic Galerkin FEM (ASGFEM, see e.g.~\cite{EPS17,EMPS18}) and in particular the sampling-free Bayesian inversion presented in~\cite{EMS18} where the setting of uniform random variables was examined.
  A generalization to the important case of Gaussian random variables turns out to be non-trivial from a computational point of view due to the difficulties of representing highly concentrated densities in a compressing tensor format which is required in order to cope with the high dimensionality of the problem.
  As a consequence, we develop a discretization approach which takes into account the potentially problematic structure of the probability density at hand by a combination of several transformations and approximations that can be chosen adaptively to counteract the interplay of the employed numerical approximations.
  With the computed functional representation of the density, the evaluation of moments or other statistical quantities of interest can be carried out efficiently and with high accuracy.
  
  The central idea of the method is to obtain a map which transports the target density to some convenient reference density and employ low-rank regression techniques to obtain a functional representation, for which accurate numerical methods are available.
  Transport maps for probability densities are a classical topic in mathematics, cf.\cite{villani2008optimal,santambrogio2015optimal}.
  They are under active research in particular in the area of optimal transport~\cite{villani2008optimal,santambrogio2015optimal} and also have become popular in current machine learning research~\cite{tran2019discrete,rezende2015variational,detommaso2019hint}.
  A main application we have in mind is Bayesian inversion where, given a prior density and some observations of the forward model, a posterior density 
  should be determined.
  In this context, the rescaling approaches in~\cite{schillings2016scaling,schillings2019convergence} based on the Laplace approximation can be considered as transport maps of a certain (affine) form.
  More general transport maps have been examined extensively in~\cite{el2012bayesian,parno2018transport} and other works of the research group.
  Obtaining a transport map is in general realized by minimizing a certain loss functional, \eg the Kullback-Leibler distance, between the target and a pushed-forward reference density. 
  This process has been analyzed and improved using iterative maps~\cite{bigoni2019greedy} or multi-scale approaches~\cite{parno2016multiscale}.
  However, the optimization, the loss functional and the chosen model class for the transport map yield only an approximation to an \emph{exact} transport. 
  We hence suppose that, in general, an inexact transport is available.
  By a pull-back argument, this can be interpreted as starting from a different or slightly perturbed reference density.
  The degree of perturbation has then to be coped with in subsequent approximation steps to enable an explicit representation of this new reference and make quantities of interest (QoI) directly accessible.
  Finding a suitable approximation relies on concepts from adaptive finite element methods (FEM).
  In addition to the selection of (local) approximation spaces of a certain degree(``p-refinement''), we introduce a spatial decomposition of the density representation into layers (``h-refinement'') around some center of mass of the considered density.
  This enables to exploit the decay behavior of the approximated density.
  Overall, this ``hp-refinement'' allows to balance inaccuracies and hence perturbations of the reference density by putting more effort into the discretization part.
  One hence has the freedom to decide whether more effort should be invested into computing an exact transport map or into a more elaborate discretization (with more layers and larger basis) of the perturbed reference density.
  
  For eventual computations with the devised (possibly high-dimensional) functional density representation, an efficient representation format is required.
  In our context, hierarchical tensors and in particular tensor trains (TT) prove to be advantageous, cf.~\cite{bachmayr2016tensor,oseledets2011tensor}.
  These low-rank formats enable to alleviate the curse of dimensionality under suitable conditions and allow for efficient evaluations of high-dimensional objects.
  For each layer of the discretization we aim to obtain a low-rank tensor representation of the respective perturbed reference density.
  In certain ideal cases, such as transporting to the standard Gaussian density, a rank-one representation is sufficient.
  Having a perturbed reference density that is Gaussian but not standard normal, the theory in~\cite{rohrbach2020rank} applies.
  In more general cases, a low-rank representability may be observed numerically.
  To allow for tensor methods to be applicable, the desired discretization layers have to be tensor domains.
  Therefore, the underlying perturbed reference density is transformed to an alternative coordinate system which benefits the representation and allows to exploit the regularity and decay behavior of the density. 
  To generate a tensor train representation (coupled with a function basis which is then also called extended or functional TT format~\cite{gorodetsky2015function}), the Variational Monte Carlo (VMC) method~\cite{ESTW19} is employed.
  It basically is a tensor regression approach based on function samples for which a convergence analysis is available.
  Notably, depending on the chosen loss functional, it leads to the best approximation in the respective model space.
  It has previously been examined in the context of random PDEs in~\cite{ESTW19} as an alternative nonintrusive numerical approach to Stochastic Galerkin FEM in the TT format~\cite{EPS17,EMPS18}.
  The approximation of~\cite{EMPS18} is used in one of the presented examples for Bayesian inversion with the random Darcy equation with lognormal coefficient.
  We note that surrogate models of the forward model have been used in the context of MCMC e.g. in~ \cite{li2014adaptive} and tensor representations (obtained by cross approximation) were used in~\cite{dolgov2018approximation} to improve the efficiency of MCMC sampling.
  
  The derivation of our method is supported by an a priori convergence analysis with respect to the Hellinger distance and the Kullback-Leibler divergence.
  In the analysis, different error sources have to be considered, in particular a layer truncation error depending on decay properties of the density, a low-rank truncation error and model space approximations are introduced.
  Moreover, the VMC error analysis~\cite{ESTW19} comprising statistical estimation and numerical approximation errors is adjusted to be applicable to the devised approach.
  While not usable for an a posterior error control in its current initial form, the derived analysis leads the way to more elaborate results for this promising method in future research.
  
  With the constructed functional density surrogate, sampling-free computations of statistical quantities of interest such as moments or marginals become feasible by fast tensor contractions, even for highly concentrated or (depending on the available transport map) nonlinearly transformed high-dimensional densities.
  
  While several assumptions have to be satisfied for this method to work most efficiently, the approach is rather general and can be further adapted to the considered problem.
  Moreover, it should be emphasized that by constructing a functional representation, structural properties of the density at hand (in particular smoothness, sparsity, low-rank approximability and decay behavior in different parameters) can be exploited in a much more extensive way than what is possible with sampling based methods such as MCMC, leading to more accurate statistical computations and better convergence rates.
  We note that the perturbed posterior surrogate can be used to efficiently generate samples by rejection sampling or within a MCMC scheme.
  Since the perturbed transport can be seen as a preconditioner, the sample generation can be based on the perturbed prior.
  These samples can then be pushed forward to the posterior.
  As a prospective extension, the constructed posterior density could directly be used in a Stochastic Galerkin FEM based on the integral structure, closing the loop of forward and inverse problem, resulting in the inferred forward problem with model data determined by Bayesian inversion from the observed data.
  
  \bigskip
  
  The structure of the paper is as follows.
  Section~\ref{sec:DensityRepresentation} is concerned with the representation of probability densities and introduces a relation between a target and a reference density.
  Such a transport map can be determined numerically by approximation in a chosen class of functions and with an assumed structure, leading to the concept of perturbed reference densities.
  To counteract the perturbation, a layered truncated discretization is introduced.
  An efficient low-rank representation of the mappings is described in Section~\ref{sec:low-rank} where the tensor train format is discussed.
  In order to obtain this nonintrusively, the Variational Monte Carlo (VMC) tensor reconstruction is reviewed.
  A priori convergence results with respect to the Hellinger distance and Kullback-Leibler divergence are derived in Section~\ref{sec:error estimates}.
  For practical purposes, the proposed method is described in terms of an algorithm in Section~\ref{sec:Algorithm}.
  Possible applications we have in mind are examined in Section~\ref{sec:applications}.
  In particular, the setting of Bayesian inverse problems is recalled.
  Moreover, the computation of moments and marginals is scrutinized.
  Section~\ref{sec:numerical experiments} illustrates the performance of the proposed method.
  In addition to an examination of the numerical sensitivity of the accuracy with respect to the perturbation of the transport maps, a typical model problem from Uncertainty Quantification (UQ) is depicted, namely the identification of a parametrization for the random Darcy equation with lognormal coefficient given as solution of a stochastic Galerkin FEM.
  
  \section{Density representation}
  \label{sec:DensityRepresentation}
  
  The aim of this section is to introduce the central ideas of the proposed approximation of densities.
  For this task, two established concepts are reviewed, namely \emph{transport maps}~\cite{el2012bayesian,bigoni2019greedy}, which are closely related to the notion of optimal transport~\cite{villani2008optimal,santambrogio2015optimal}, and \emph{hierarchical low-rank tensor representations}~\cite{oseledets2011tensor,hackbusch2012tensor,bachmayr2016tensor}.
  By the combination of these techniques, assuming the access to a suitable transformation,
  the developed approach yields a functional representation of the density in a format which is suited to computations with high-dimensional functions.
  In particular, we are able to handle highly concentrated posterior densities, \eg appearing in the context of Bayesian inverse problems.
  While transport maps on their own in principle enable the generation of samples of some target distribution, the combination with a functional low-rank representation allows for integral quantities such as (centered) moments to become computable.
  Given an approximate transport map, the low-rank representation can be seen as a further approximation step (improving the inaccuracy of the used transport) to gain direct access to the target density.

  Consider a target measure $\pi$ with Radon-Nikodym derivative with respect to the Lebesgue measure $\lambda$ denoted as $f$ with support in $\mathbb{R}^d$, $d<\infty$, \ie
  \begin{equation}
  \label{eq:main density}
  f(y) := \frac{\mathrm{d}\pi}{\mathrm{d}\lambda}(y),\quad y\in Y:=\mathbb{R}^d.
  \end{equation}
  In the following we assume that point evaluations of $f$ are available up to a multiplicative constant, motivated by the framework of Bayesian posterior density representation with unknown normalization constant.
  Furthermore, let $\pi_0$ be some reference measure exhibiting a Radon-Nikodym derivative with respect to to the Lebesgue measure denoted as $f_0$.
  This is motivated by the prior measure and density in the context of Bayesian inference.

  \subsection{Transport Maps}
  \label{sec:transportMaps}
  
  The notion of density transport is classical and with optimal transport has become a popular field recently, see \eg~\cite{villani2008optimal,santambrogio2015optimal}.
  It has been employed to improve numerical approaches for Bayesian inverse problems for instance in~\cite{el2012bayesian,bigoni2019greedy,dolgov2018approximation}.
  Similar approaches are discussed in terms of sample transport \eg for Stein's method~\cite{liu2016stein,detommaso2018stein} or multi-layer maps~\cite{bigoni2019greedy}.
  We review the properties required for our approach in what follows.
  Note that since our target application is Bayesian inversion, we usually use the terms ``prior'' and ``posterior'' instead of the more general ``reference'' and ``target'' densities.

  Let $X:=\mathbb{R}^d$ and assume that there exists an exact transport map 
  \begin{equation}
  \label{eq:exact_transport}
      T\colon X\to Y,
  \end{equation} which is a diffeomorphism\footnote{note that the requirements on $T$ can be weakened, \eg to local Lipschitz} that relates $\pi$ and $\pi_0$ via pullback, \ie
  \begin{align}
      \label{eq:exact_prior}
      f_0(x) = f(T(x)) |\det\mathcal{J}_T(x)|, \quad x\in X.
  \end{align}
  Then, computations might be carried out in terms of the measure $\pi_0$, which is commonly assumed to be of a simpler structure.
  For instance the moment computation with respect to some multiindex ${\bm{\alpha}}$ reads as follows,
  \begin{align}
      \int\limits_{Y} y^{\bm{\alpha}}\mathrm{d}\pi(y) 
      &= 
      \int\limits_{X} T(x)^{\bm{\alpha}} \mathrm{d}\pi_0(x) \nonumber\\
      &= 
      \int\limits_{X} T(x)^{\bm{\alpha}} f_0(x)\mathrm{d}\lambda(x).
      \label{eq:transport_moment_equation}
  \end{align}
  Note that the computation of the right-hand side in \eqref{eq:transport_moment_equation} may still be a challenging task depending on the actual structure of $T$.
  In~\cite{el2012bayesian} $T$ is expanded in chaos polynomials with respect to $\pi_0$.
  From a practical point of view, this provides access to lower-order moments using orthogonality of the underlying polynomial system.

  Here we follow an alternative strategy with the aim to efficiently compute moments of some target density based on a functional representation.
  Notably we assume a convenient (simple) structures of $T$ with the potential drawback of reduced accuracy, \ie an inexact (pull-back) transport from the target to an auxiliary density (instead of the exact reference).
  Motivated by the Bayesian context, we call such a pull-back of some posterior density the \emph{perturbed prior density}, see Section~\ref{sec:inexact_transport}.
  Given a simple transport structure, the possibly demanding computational task is shifted to the accurate approximation of the perturbed prior.
  For this, there is justified hope of feasibility in some appropriate (alternative) coordinate system.
  In order to tackle moment computations, other posterior statistics or to generate posterior samples, we hence devise a numerical approach that enables a \emph{workload balancing between the reconstruction of some problem-dependent transport structure and the accurate evaluation of the perturbed prior}.
  In the following we list some examples of transport maps.
  
  \subsubsection{Affine transport}
  \label{sec:affine transport}
  In~\cite{schillings2016scaling,schillings2019convergence} the authors employ an affine linear preconditioning for acceleration of MCMC or sparse-grid integration in the context of  highly informative and concentrated Bayesian posterior densities, using a s.p.d. matrix $H\in\mathbb{R}^{d,d}$ and $M\in\mathbb{R}^d$. 
  In the mentioned articles, up to a multiplicative constant, $H$ corresponds to the inverse square root of the Hessian at the MAP (maximum a posteriori probability) $M$, \ie the location of the local optimum of an Laplace approximation of the posterior density.
  This rather simple construction, under the assumption of an unimodal density, leads to stable numerical algorithms for the computation of quantities of interest as the posterior mass concentrates.
  When considering the push-forward of a reference density $f_0$ to a target density $f$ this concept coincides with an affine transport 
  \begin{equation}
      y = T(x) = Hx + M, \quad x \in X.
  \end{equation}
  In the transport setting $H$ and $M$ may be computed for instance via some minimization of the Kullback-Leibler divergence as in~\cite{el2012bayesian}.
  Note that $H$ and $M$ do not necessarily have to be the inverse square root of the Hessian or the MAP.
  Figure~\ref{fig:affine density transport} illustrates the concept of an affine transport. 
  
  \begin{figure}
    \begin{center}
    \begin{tikzpicture}
    \node at (0,0) {$\phantom{1}$};
    \foreach[evaluate=\c using int(\l-5)] \l in  {5,20,35,50,65,80,95}{
        \def\opac{1-0.005*\l}
        \draw[lightgray!\c!black, thick, opacity = \opac, rotate around={-30:(2,0)}] (2,0) ellipse ({0.005*\l} and {0.02*\l});
    }
    
    \foreach[evaluate=\c using int(\l-5)] \l in {5,20,35,50,65,80,95}{ 
        \def\opac{1-0.005*\l}
        \draw[lightgray!\c!black, thick, opacity = \opac] (9,0) circle(0.02*\l);
    }
    
    \draw [thick,<-] (3,-1) to [out=-25,in=205] (7,-1);
    \node at (5,-1) {$T(x)$};
    
    \node at (6.5,0) {$f_0$};
    \node at (3.5,0) {$f$};
    
    \end{tikzpicture}
    \caption{Illustration of affine transport: translation, rotation and rescaling.}
     \label{fig:affine density transport}
    \end{center}
    \end{figure}
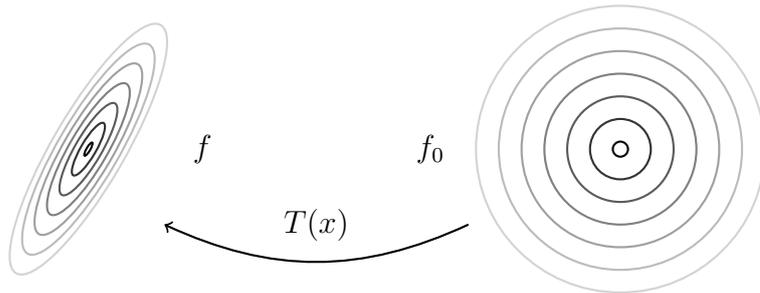

  \subsubsection{Quadratic transport}
  \label{sec:quadratic transport}
  A more general class of polynomial transport exhibits the form 
  \begin{equation}
      T(x) = \frac{1}{2}x: \bm{A} : x + Hx + M, \quad x\in X,
  \end{equation}
  with $\bm{A}\in\mathbb{R}^{d,d,d}, H\in\mathbb{R}^{d,d}, M\in\mathbb{R}^d$.
  Such a quadratic transport may be used for simple nonlinear transformations as depicted in Figure~\ref{figure:quadratic_transport}.
  \begin{center}
    \begin{figure}
      \begin{center}
      \includegraphics[width=.8\linewidth]{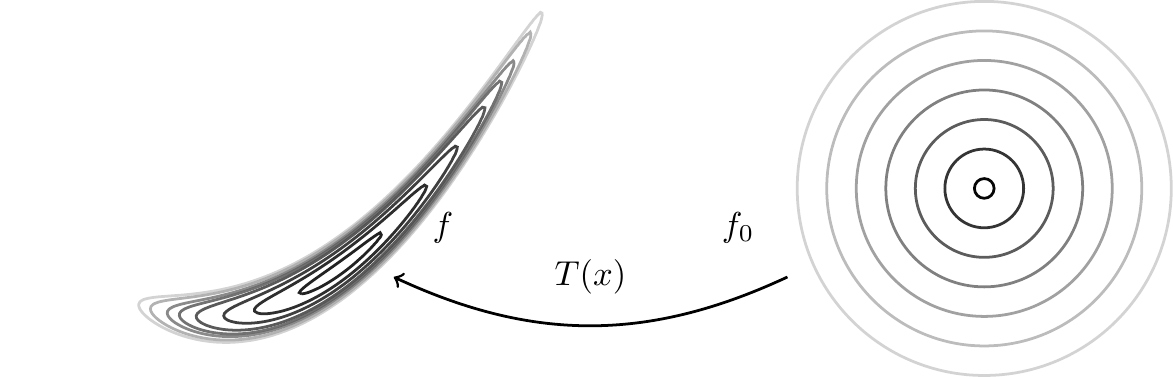}
      
          
          
          

      \caption{Illustration of quadratic transport: affine properties and bending.}
      \label{figure:quadratic_transport}
      \end{center}
      \end{figure}
  \end{center}
  \subsubsection{More general transport maps}
  The parametrization of transport maps can be chosen quite liberally as long as certain criteria are satisfied, which are either directly imposed in the ansatz space $\mathcal T$ of the maps or added as constraints during optimization.
  In particular, the approximate transport map has to be invertible, which can be ensured by requiring a positive Jacobian.
  A commonly used measure for transport optimization is the Kullback-Leibler divergence\footnote{although in machine learning Wasserstein or Sinkhorn distances have become very popular when so-called normalizing flows are computed} leading to the optimization problem
  \begin{equation}
  \min_{T\in\mathcal T} \mathrm{d}_{\mathrm{KL}}(Y;T\pi_0,\pi)\quad\text{s.t.}\quad \mathrm{det} \nabla T > 0 \qquad \text{$\pi$-a.e.}
  \end{equation}
  Several suggestions regarding simplifications and special choices of function spaces $\mathcal T$ such as smooth triangular maps based on higher-order polynomials or radial basis functions can for instance be found in the review article~\cite{el2012bayesian}.
  An interesting idea is to subdivide the task into the iterative computation of simple correction maps which are then composed as proposed in~\cite{bigoni2019greedy}.
  We again emphasize that while an accurate transport map is desirable, any approximation of such a map can in principle be used with the proposed method.
  In fact one can decide whether it is beneficial to spend more effort on the approximation of the perturbed density or on a better representation of the transport.

  \subsection{Inexact transport and the perturbed prior}\label{sec:inexact_transport}
  In general, the transport map $T$ is unknown or difficult to determine and hence has to be approximated by some $\tilde{T}\colon X\to Y$, \eg using a polynomial chaos representation with respect to $\pi_0$ \cite{el2012bayesian} or with a more advanced composition of simple maps in a reduced space such as in~\cite{bigoni2019greedy}. 
  As a consequence, it holds 
  \begin{equation}
      \int\limits_{Y} y^{\bm{\alpha}} \mathrm{d}\pi(y) 
      \approx
      \int\limits_{X} \tilde{T}(x)^{\bm{\alpha}}\mathrm{d}\pi_0(x)
  \end{equation}
  subject to the accuracy of the involved approximation of $T$.
  One can also view $\tilde{T}$ as the push-forward of some measure $\tilde{\pi}_0$ with density $\tilde{f}_0$ to $\pi$ given by
  \begin{equation}
     \label{eq:pert_prior_density}
      \tilde{f}_0(x) = f(\tilde{T}(x))|\operatorname{det}\mathcal{J}_{\tilde{T}}(x)|.
  \end{equation} 
  We henceforth refer to~\eqref{eq:pert_prior_density} as the auxiliary reference or \emph{perturbed prior density}.
  Using this construction, the moment computation reads
  \begin{equation}
  \label{eq:pert_pi_0}
      \int\limits_{Y} y^{\bm{\alpha}}\mathrm{d}\pi(y)= 
      \int\limits_{X} \tilde{T}(x)^{\bm{\alpha}}\mathrm{d}\tilde{\pi}_0
      = \int\limits_{X} \tilde{T}(x)^{\bm{\alpha}}\tilde{f}_0(x)\mathrm{d}\lambda(x)
      .
  \end{equation}
  If one would know $\tilde{f}_0$, by~\eqref{eq:pert_prior_density} and~\eqref{eq:pert_pi_0} one would also have access to the exact posterior.
  
  Equation~\eqref{eq:pert_pi_0} is the starting point of the proposed method by approximating $\tilde{f}_0$ in another coordinate system which is better adapted to the structure of the approximate (perturbed) prior.
  Consider a (fixed) diffeomorphism 
  \begin{equation}
     \label{eq:referencetrafo_map}
      \Phi\colon \hat{X}\subset\mathbb{R}^d\to X,\quad \hat{x}\mapsto x=\Phi(\hat{x}) 
  \end{equation}
  with Jacobian $\hat{x}\mapsto|\operatorname{det}\mathcal{J}_\Phi(\hat{x})|$ and define the \textit{perturbed transformed prior}
  \begin{equation}
     \label{eq:tildef0_trafo}
     \hat{f}_0\colon \hat{X}\mapsto\mathbb{R}_+,\quad \hat{x}\mapsto \hat{f}_0(\hat{x}):=\tilde{f}_0(\Phi(\hat{x})).
  \end{equation}
  In case \eqref{eq:tildef0_trafo} can be approximated accurately by some function $\tilde{f}_0^h$ then 
  \begin{equation}
      \label{eq:tildef0TTmoment}
      \int\limits_{Y} y^{\bm{\alpha}}\mathrm{d}\pi(y)
      \approx
      \int\limits_{\hat{X}} \tilde{T}(\Phi(\hat{x}))^{\bm{\alpha}} \tilde{f}_0^h(\hat{x})|\operatorname{det}\mathcal{J}_\Phi(\hat{x})| \mathrm{d}\lambda(\hat{x}) 
  \end{equation}
  with accuracy determined only by the approximation quality of $\tilde{f}_0^h$.
  Thus, \eqref{eq:tildef0_trafo} and~\eqref{eq:tildef0TTmoment} enable a balancing between the construction of the transport map approximation $\tilde{T}$ of $T$ to shift its complexity given the underlying diffeomorphism $\Phi$ to the approximation of~\eqref{eq:tildef0_trafo} in a new coordinate system intrinsic to $\hat{X}$.

  The construction of $\tilde{T}$ and a suitable map in~\eqref{eq:referencetrafo_map} may be used to obtain a convenient transformed auxiliary reference density given in~\eqref{eq:tildef0_trafo}.
  An approximation thereof can be significantly simpler compared to a possibly complicated and concentrated target density $f$ or the computation of the exact transport $T$.
  This \eg is satisfied if
  \begin{itemize}
      \item $f_0$ is a Gaussian density and $\tilde{T}$ maps $f$ to $\tilde{f}_0$ which is in some sense near to a Gaussian density.
      In this case, $\Phi$ from~\eqref{eq:referencetrafo_map} may be chosen as the $d$-dimensional spherical transformation and extended low-rank tensor formats are employed for the approximation, see Section~\ref{sec:low-rank}.
      In this setting, the introduction of an adapted coordinate system allows to shift the exponential decay to the one dimensional radial parameter.
      The accuracy of an approximation can then be improved easily by additional $h$-refinements as described in Section~\ref{sec:layer}.
      \item The reference density $f_0$ has a complicated form and might be replaced by $\tilde{f}_0$ to become computationally accessible.
  \end{itemize} 
  
  In the following we state an important property that needs to be fulfilled by the perturbed prior $\tilde{f}_0$ in order to lead to a convergent method with the employed approximations.
  
  \begin{definition}(\textbf{outer polynomial exponential decay})
  A function $\tilde{f}_0\colon X\to\mathbb{R}^+$ has outer polynomial exponential decay if there exists a simply connected compact $K\subset X$ with a polynomial $\pi^+$ being positive on $X\setminus K$ and some $C>0$ such that 
  \begin{equation}
  \tilde{f}_0(x) \leq C \exp{(-\pi^+(x))},\quad x\in X\setminus{K}.
  \end{equation}
  \end{definition}
  
  \subsection{Layer based representation}
  \label{sec:layer}
  To further refine and motivate the notion of an adapted coordinate system, let $L\in\mathbb{N}$ and $(X^\ell)_{\ell=1}^{L}$ be pairwise disjoint domains in $X$ s.t. 
  \begin{equation}
      K:=\bigcup_{\ell=1}^L \overline{X^\ell}
  \end{equation}
  is simply connected and compact and define $X^{L+1}:= X\setminus{K}$.
  Then, for given $L\in\mathbb{N}$ we may decompose the perturbed prior $\tilde{f}_0$ as
  \begin{equation}
  \label{eq:density_decomposition}
      \tilde{f}_0(x) = \sum\limits_{\ell=1}^{L+1} \tilde{f_0^\ell}(x) \quad\text{with}\quad \tilde{f_0^\ell} := \chi_{\ell}\tilde{f}_0,
  \end{equation}
  where $\chi_{\ell}$ denotes the indicator function on $X^\ell$. 
  Moreover, for any tensor set $\hat{X}^\ell := \bigtimes_{i=1}^d\hat{X}_i^\ell$ and diffeomorphism $\Phi^\ell\colon \hat{X}^\ell\mapsto X^\ell$, $1\leq \ell\leq L+1$, we may represent the \emph{localized perturbed prior} $\tilde{f_0}^\ell$ as a pull-back function 
  \begin{equation}
      \label{eq:local_pert_prior}
      \tilde{f_0^\ell} = \hat{f_0^\ell}\circ {\Phi^{\ell}}^{-1},
  \end{equation}  
  where $\hat{f_0^\ell}$ is a map defined on $\hat{X}^\ell$ as in~\eqref{eq:tildef0_trafo}.
  We consider the following example.
  \begin{example}
    \label{ex:polar_coords}(\textbf{multivariate polar transformation})\\
    The d-dimensional spherical coordinate system allows for simple layer layouts in terms of hyperspherical shells.
    In particular, for $\ell=1,\ldots,L+1<\infty$, with $0=\rho_1<\rho_2<\ldots< \rho_{L+1} < \rho_{L+2} = \infty$, let
    \begin{align*}
        \hat{X}^\ell&:=[\rho_\ell,\rho_{\ell+1}]\times[0,2\pi]\times\bigtimes_{i=2}^{d-2}[0,\pi],\\ 
        X^\ell&:= B_{\rho_{\ell+1}}(0)\setminus B_{\rho_{\ell}}(0)\subset X,
    \end{align*} 
    \ie $\hat{X}^\ell$ and $X^\ell$ denote the corresponding adopted (transformed) and the original parameter space, respectively.
    Then, for $\hat{x} = (\rho, \theta_0, \bm{\theta})\in\hat{X}$, $\bm{\theta}=(\theta_1, \ldots, \theta_{d-2})$, the polar transformation $\Phi^\ell\colon \hat{X}^\ell \to X^\ell$  reads
     \begin{equation}
    \Phi^\ell (\hat{x}) 
          = \rho
    \left[
    \begin{array}{r}
         \cos\theta_0\sin\theta_1\sin\theta_2\cdots \sin\theta_{d-3}\sin\theta_{d-2} \\
         \sin\theta_0\sin\theta_1\sin\theta_2\cdots \sin\theta_{d-3}\sin\theta_{d-2} \\
         \cos\theta_1\sin\theta_2\cdots \sin\theta_{d-3}\sin\theta_{d-2}\\
         \cos\theta_2 \cdots \sin\theta_{d-3}\sin\theta_{d-2}\\
         \vdots\\
         \cos\theta_{d-3}\sin\theta_{d-2}\\
         \cos\theta_{d-2}
    \end{array}
    \right].
    \end{equation}
    Moreover, the Jacobian is given by
    \begin{align}
      \det \mathcal{J}_{\Phi^\ell}(\rho, \theta_0, \bm{\theta})=\rho^{d-1} \prod\limits_{i=1}^{d-2}\sin^{i}\theta_i.
    \end{align}
  \end{example}
  This layer based coordinate change enables a representation of the density on bounded domains.
  Even though the remainder layer is unbounded, we assume that $K$ is sufficiently large to cover all probability mass of $\tilde{f}_0$ except for a negligible higher-order error.

  Up to this point, the choice of transformation $\Phi^\ell$, $\ell=1, \ldots, L+1$, is fairly general. 
  However, for the further development of the method we assume the following property.
  
  \begin{definition}(\textbf{rank 1 stability})\\
  Let $\mcX, \hat{\mcX} = \bigtimes_{i=1}^d \hat{\mcX}_i\subset\mathbb{R}^d$ be open and bounded sets. 
  A diffeomorphism $\Phi\colon\hat{\mcX}\mapsto \mcX$ is called \emph{rank 1 stable} if $\Phi$ and the absolute value of its Jacobian $\det \mathcal{J}_\Phi$ have rank 1, \ie there exists univariate functions $\Phi_i\colon \hat{\mcX}_i \to \mcX$ , $h_i\colon \hat{\mcX}\to\bbR$, $i=1,\ldots, d$, such that for $\hat{x}\in\hat{\mcX}$
  \begin{equation}
    \label{eq:def:rank1}
    \Phi(\hat{x}) = \prod_{i=1}^d \Phi_i(\hat{x}_i), \quad \abs{\det\mathcal{J}_{\Phi}(\hat{x})} = \prod_{i=1}^d h_i(\hat{x}_i).
  \end{equation}
  \end{definition}
  
  \begin{proposition}
  The multivariate polar coordinate transformation from Example~\ref{ex:polar_coords} is rank 1 stable.
  \end{proposition}
  Due to the notion of rank 1 stable transformations, the map $\hat{x}\mapsto T(\Phi(\hat{x}))$ in~\eqref{eq:tildef0TTmoment} inherits the rank structure of $T$, see Section~\ref{sec:low-rank}.
  Furthermore, since the Jacobian $\hat{x}\mapsto |\operatorname{det}\mathcal{J}_\Phi(\hat{x})|$ is rank 1, we can construct tensorized orthonormal basis functions which may be used to approximate the perturbed transformed prior in~\eqref{eq:tildef0_trafo}.
  
  \begin{remark}
  The described concept can be extended to any rank $r\in\mathbb{N}$ Jacobians of $\Phi$, \ie 
  \begin{equation}
      |\operatorname{det}\mathcal{J}_\Phi(\hat{x})| = \sum\limits_{k=1}^r\prod\limits_{i=1}^d h_{i,k}(\hat{x}_i).
  \end{equation}
  Motivated by the right-hand side in \eqref{eq:tildef0TTmoment}, one may use different approximations of the perturbed transformed prior $\tilde{f}_0\circ\Phi$ in $r$ distinct tensorized spaces, each associated to the rank $1$ weight $\prod\limits_{i=1}^d h_{i,k}$.
  \end{remark}

  \subsection{Layer truncation}
  This paragraph is devoted to the treatment of the last (remainder or truncation) layer introduced in~\eqref{eq:density_decomposition} with the aim to suggest some approximation choices.
  
  If $\tilde{f_0}$ is represented in the layer format~\eqref{eq:density_decomposition}, it is convenient to simply extend the function to zero after layer $L\in\mathbb{N}$.
  By this, the remaining (possibly small) probability mass is neglected.
  Such a procedure is typically employed in numerical applications and does not impose any computational issues since events on the outer truncated domain are usually exponentially unlikely for truncation value chosen sufficiently large.
  Nevertheless, in order to present a rigorous treatment, we require properties like absolute continuity, which would be lost by using a cut-off function.
  Inspired by~\cite{schillings2019convergence} regarding the information limit of unimodal posterior densities\footnote{A result of~\cite{schillings2019convergence} is that under suitable conditions the posterior distribution converges to a Gaussian in the limit of zero noise and infinite measurements.}, we suggest a Gaussian approximation for the last layer $L+1$ on the unbounded domain $X^{L+1}$, \ie for some s.p.d. $\Sigma\in\mathbb{R}^{d, d}$ and $\mu\in\mathbb{R}^d$ we define the \emph{hybrid representation of the perturbed prior} by
  \begin{equation}
      \label{eq:layer representation + split}
      \tilde{f}_0^{\on{Trun}}(x) :=
      C_L
      \left\{
      \begin{array}{ll}
      \tilde{f}_0^\ell(x), & x\in X^\ell, \ell = 1,\ldots,L, \\
      f_{\Sigma,\mu}(x), & x \in X^{L+1},
      \end{array}
      \right.
  \end{equation}
  with $C_L = (C_L^{<}+ C_L^{>})^{-1}$, where
  \begin{align}
      \label{eq:norm constant<}
      C_{L}^{<} &:= \int\limits_{X \setminus{K}} f_{\Sigma,\mu}(x)\,\mathrm{d}\lambda(x), 
      \\
      \label{eq:norm constant>}
      C_L^{>} &:= \sum\limits_{\ell=1}^L \int\limits_{X^\ell} \tilde{f}_0^\ell(x)\,\mathrm{d}\lambda(x),
  \end{align}
  and $f_{\Sigma, \mu}$ denotes the Gaussian probability density function with mean $\mu$ and covariance matrix $\Sigma$.

  \begin{remark} 
      \label{rem:affine trafo choice}
  A good choice for $\mu$ and $\Sigma$ would be the mean and covariance of the exact perturbed prior $\tilde{f}_0$, which however is not accessible a priori. 
  Thus, in numerical simulations one may choose $\mu$ and $\Sigma$ as (centralized) moments of the normalized truncated perturbed prior density $\tilde{f_0^{\on{Trun}}}|_K$ or as the MAP point and the corresponding square root of the numerically computed Hessian as a result of an optimization algorithm on $\tilde{f}_0$.
  \end{remark}
  
  \begin{lemma}(\textbf{truncation error}) 
  \label{proposition:truncation}
  For $\mu\in\bbR^d$ and $\Sigma\in\bbR^{d, d}$ let $\tilde{f}_0$ have outer polynomial exponential decay with positive polynomial $\tilde{\pi}^+$ and $\tilde{C}>0$ with $K=\overline{B_R(\mu)}$ for some $R>0$. 
  Then, for $C_\Sigma = 1/2\lambda_{\mathrm{min}}(\Sigma^{-1})$ there exists $C=C(\tilde{C},\Sigma,d, C_\Sigma)>0$ such that
  \begin{align*}
       \norm{\tilde{f}_0 - \tilde{f}_0^{\on{Trun}}}_{L^1(X\setminus K)} \lesssim
       \norm{\exp{(-\tilde{\pi}^+)}}_{L^1(X\setminus K)} +\\
       \Gamma\left(d/2, C_\Sigma R^2\right)
  \end{align*}
  and 
  \begin{align*}
  &\left|\, \int\limits_{X\setminus{K}} \log\left( \frac{\tilde{f}_0}{f_{\Sigma,\mu}} \right) \tilde{f}_0\,\mathrm{d}x\right|  \leq \\ & \qquad
  \int\limits_{X\setminus K} \left(\frac{1}{2}\|x\|_{\Sigma^{-1}}^2 + \tilde{\pi}^+(x)\right)e^{-\tilde{\pi}^+(x)}\,\mathrm{d}\lambda(x)
  \end{align*}
  with the incomplete Gamma function $\Gamma$.
  \end{lemma}
  \begin{proof}
    The proof follows immediately from the definition of $\tilde{f}_0^{\on{Trun}}$.
  \end{proof}
  In the case that the perturbed prior is close to a Gaussian standard normal distribution, it holds $c\approx 1$. 
  
  Note that the constant $C_{L}^{<}$ in~\eqref{eq:norm constant<} may exhibit an analytic form whereas computing~$C_{L}^{>}$ suffers from the curse of dimensionality and is in general not available.
  To circumvent this issue and render further use of the representation~\eqref{eq:layer representation + split} feasible, we introduce a suitable low-rank approximation in the next section.
  
  \section{Low-rank tensor train format}
\label{sec:low-rank}
The computation of high-dimensional integrals and the efficient construction of surrogates is a challenging task with a multitude of approaches.
Some of these techniques are sparse grid methods~\cite{chen2016sparse,garcke2012sparse}, collocation~\cite{ernst2019expansions,nobile2008sparse,foo2010multi} or modern sampling techniques~\cite{gilks1995markov,rudolf2017metropolis,neal2001annealed}.
As motivated by $C_{L}^{>}$ in~\eqref{eq:norm constant>}, we aim for a model to adequately approximate the localized perturbed prior maps $\tilde{f}_0^\ell$.
The introduction of an adapted coordinate system enables the use of low-rank representations such as the tensor train (TT) format~\cite{oseledets2011tensor,holtz2012alternating,hackbusch2012tensor} described in this section.
We highlight a ``non-intrusive'' sample-based technique to obtain such a representation of arbitrary maps, namely the Variational Monte Carlo (VMC) method~\cite{ESTW19}.

Let $\hat{X} = \bigotimes_{i=1}^d \hat{X_i}$ be a tensor space of separable Banach spaces $\hat X_i$, $i\in[d] \isdef \{1, \ldots, d\}$, and consider a map $g\colon \hat{X} \to \mathbb{R}$.
The function $g$ can be represented in the TT format if there exists a \emph{rank vector} $\bm{r}=(r_1, \ldots, r_{d-1})\in\bbN^{d-1}$ and univariate functions $g^i[k_{i-1}, k_i]\colon \hat{X}_i\to\mathbb{R}$ for $k_i\in[r_i]$, $i\in[d]$, such that for all $\hat x\in\hat{X}$
\begin{equation}
    \label{eq:tt representation}
    g(\hat x) = \sum_{\bm{k}=\bm{1}}^{\bm{r}} \prod_{i=1}^d g^i[k_{i-1}, k_i](\hat x_i), \quad \bm{k}\isdef (k_1, \ldots, k_{d-1}).
\end{equation}
For ease of notation it is convenient to set $k_0 = k_d = 1$.
In the forthcoming sections we consider weighted tensorized Lebesgue spaces.
In particular, for a non-negative weight function $w\colon\hat{X}\to\bbR$ with $w = \bigotimes_{i=1}^d w_i$, $w\in L^1(\hat{X})$, define the tensorization of $L^2(\hat{X}, w) = \bigotimes_{i=1}^d L^2(\hat{X}_i, w_i)$ by
\begin{align}
 \label{eq:tensorl2}
 &\mathcal{V}(\hat{X}) := L^2(\hat{X},w)  \nonumber \\
 &=\left\{ v\colon \hat{X}\to\bbR\;\vert\; \|v\|_{\mathcal{V}}^2:=\int_{\hat{X}} v(\hat{x})^2 w(\hat{x})\,\mathrm{d}\lambda(\hat{x}) < \infty \right\}.
\end{align}
We assume that there exists an complete orthonormal basis $\{P_k^i:k\in\mathbb{N}\}$ in $L^2(\hat{X}_i, w_i)$ for every $i\in [d]$ which is known a priori.
For discretization purposes, we introduce the finite dimensional subspaces

\begin{equation}
  \label{eq:basis_span}
  \mcV_{i, n_i}:= \overline{\Span\left\{P_1^i, \ldots, P_{n_i}^i\right\}}\subseteq L^2(\hat{X}_i, w_i)
\end{equation}
for $i=1,\ldots,d,$ and $n_i\in \mathbb{N}$.
On these we formulate the \emph{extended tensor train format} in terms of the coefficient tensors 
\begin{align}
    G^i \colon \left[r_{i-1}\right]\times\left[n_{i}\right]\times\left[r_{i}\right]&\to \mathbb{R}, \nonumber
\\
  \; (k_{i-1}, j, k_i)&\mapsto G^i[k_{i-1}, j, k_i], \quad i\in\left[d\right],
\end{align}
such that every univariate function $g^i\in\mcV_{i, n_i}$ can be written as
\begin{equation}
  \label{eq:finite tt component}
    g^i[k_{i-1}, k_i](\hat x_i) = \sum_{j=1}^{n_i} G^i[k_{i-1}, j, k_i]P_j^i(\hat x_i) \quad \text{for}\; \hat x\in \hat X_i.
\end{equation}
For the full tensor format the function 
\begin{equation}
\label{eq:definition_V_Lambda}
g\in\mcV_{\Lambda}\isdef\bigotimes_{i=1}^d \mcV_{i, n_i}\subseteq \mcV(\hat{X})
\end{equation}
 can be expressed by a high dimensional algebraic tensor $G \colon \Lambda \isdef \bigtimes_{i=1}^d [n_i] \to \mathbb{R}$ and tensorized functions $P_\alpha\isdef \bigotimes_{i=1}^d P_{\alpha_i}$ for $\alpha=(\alpha_1,\ldots, \alpha_d)\in\Lambda$ such that
\begin{equation}
    \label{eq: full tensor representation}
    g(\hat x) = \sum_{\bm{\alpha}\in\Lambda} G[\alpha_1, \ldots, \alpha_d]\prod_{i=1}^d P_{\alpha_i}(\hat x_i).
\end{equation}
In contrast to this, the format given by~\eqref{eq:tt representation} and~\eqref{eq:finite tt component} admits a linear structure in the dimension.
More precisely, the memory complexity of $\mathcal{O}(\max\{n_1, \ldots, n_d\}^d)$ in~\eqref{eq: full tensor representation} reduces to 
\begin{equation}
\mathcal{O}(\max\{r_1,\ldots, r_{d-1}\}^2 \cdot d \cdot \max\{n_1,\ldots, n_d\}).
\end{equation}
This observation raises the question of expressibility for certain classes of functions and the existence of a low-rank vector $\bm{r}$ where $\max\{r_1, \ldots, r_{d-1}\}$ is sufficiently small for practical computations.
This issue is \eg addressed in~\cite{schneider2014approximation,bachmayr2017parametric,griebel2013construction} under certain assumptions on the regularity and in~\cite{espig2009black,oseledets2011tensor,ballani2013black,ESTW19} explicit (algorithmic) constructions of the format are discussed even in case that $g$ has no analytic representation.

For later reference we define the finite dimensional low-rank manifold of rank $\bm{r}$ tensor trains by
\begin{equation}
  \label{eq:VMCModelClass}
  \mathcal{M}_{\bm{r}} := \{g\in \mcV(\hat{X})\;\vert\; g \text{ as in }~(\ref{eq:tt representation})\text{ with } g^i\text{ as in }~(\ref{eq:finite tt component}) \}.
\end{equation}
This is an embedded manifold in the finite full tensor space $\mathcal{V}_{\Lambda}$ from~\eqref{eq:definition_V_Lambda} admitting the cone property.
We also require the concept of the algebraic (full) tensor space
\begin{equation}
  \mathbb{T} \isdef \left\{ G\colon \bbN^d \to\bbR \right\}
\end{equation}
and the corresponding low-rank form for given $\bs{r}\in\bbN^{d-1}$ defined by
\begin{equation}
  \mathbb{TT}_{\bs{r}} \isdef \left\{ G\colon\Lambda\to \bbR \;\vert \; G[\alpha] = \sum_{\bs{k}=\bs{1}}^{\bs{r}}\prod_{i=1}^d G[k_{i-1}, \alpha_i, k_i] \right\}.
\end{equation}

Without going into detail, we mention the higher order singular value decomposition (HOSVD), which is used to decompose a full algebraic tensor into a low-rank tensor train. 
The algorithm is based on successive unfoldings of the full tensor into matrices, which are orthogonalized and possibly truncated by a singular value decomposition, see~\cite{oseledets2010tt} for details.
This algorithm enables us to state the following Lemma.
\begin{lemma}[{\cite[Theorem 2.2]{oseledets2010tt}}]
  \label{lem:HOSVD}
  For any $g\in\mcV_{\Lambda}$ and $\bs{r}\in\bbR^{d-1}$ there exists an extended low-rank tensor train $g_{\bs{r}}\in\mathcal{M}_{\bs{r}}$ with
  \begin{equation}
  	\label{eq:hosvd error}
    \norm{g - g_{\bs{r}}}_{\mcV(\hat{X})}^2 \leq \sum_{i=1}^{d-1}\sigma_i^2,
  \end{equation}
  where $\sigma_i$ is the distance of the $i$-th unfolding matrix of the coefficient tensor of $g$ in the HOSVD to its best rank $r_i$ approximation in the Frobenius norm.
\end{lemma}
\begin{proof}
  The proof follows from the best approximation result of the usual matrix SVD with respect to the Frobenius norm and the orthonormality of the chosen basis.
\end{proof}
\begin{remark}
  Estimate~\eqref{eq:hosvd error} is rather unspecific as the $\sigma_i$ cannot be quantified a priori.
  In the special case of Gaussian densities we refer to~\cite{rohrbach2020rank} for an examination of the low-rank representation depending on the covariance structure.
  By considering a transport $\tilde{T}$ that maps the considered density only ``close'' to a standard Gaussian, the results can be applied immediately to our setting and more precise estimates are possible.
\end{remark}

\subsection{Tensor train regression by Variational Monte Carlo}
\label{section:tensor_train_regression}

We review the sampling-based VMC method presented in~\cite{ESTW19} which is employed to construct TT representations of the local maps $\Phi^\ell$ as in~\eqref{eq:local_pert_prior}.
The approach generalizes the concept of randomized tensor completion~\cite{eigel2019non} and its analysis relies on the theory of statistical learning, leading to a priori convergence results.
It can also be seen as a generalized tensor least squares technique.
An alternative cross-interpolation method for probability densities is presented in~\cite{dolgov2018approximation}.

For the VMC framework, consider the \emph{model class} $\mathcal{M}_{\bm{r}}(\underline{c},\overline{c})\subset\mathcal{M}_{\bm{r}}$ of truncated rank $\bm{r}\in\mathbb{R}^{d-1}$ tensor trains which is given for $0\leq\underline{c} < \overline{c}\leq\infty$ by
\begin{equation}
\mathcal{M}_{\bm{r}}(\underline{c},\overline{c})
\isdef
\left\{ g\in \mathcal{M}_{\bm{r}}\;|\; \underline{c} \leq g(\hat{x})\leq \overline{c} \quad\text{a.e. in}\quad \hat{X} \right\}.
\end{equation}
The model class $\mathcal{M}_{\bm{r}}(\underline{c},\overline{c})$ is a finite subset of the truncated nonlinear space $\mathcal{V}(\hat{X}, \underline{c},\overline{c})\subseteq\mcV(\hat{X})$ defined as
\begin{equation}
\mathcal{V}(\hat{X}, \underline{c},\overline{c}) := \{ v\in L^2(\hat{X},w)\;|\; \underline{c} \leq v(\hat{x})\leq \overline{c} \quad\text{a.e. in } \hat{X} \}, 
\end{equation}
which we equip with the metric $d_{\mathcal{V}(\hat{X},\underline{c},\overline{c})}(v,w):=\|v-w\|_{\mathcal{V}}$.

Alternatively, for numerical purposes we may characterize $\mathcal{M}_{\bm{r}}(\underline{c},\overline{c})$ and $\mathcal{V}(\hat{X},\underline{c},\overline{c})$ in terms of constraints on the coefficients of the underlying representation with respect to $\{P_\alpha\}$.
For $\ell^2(\mathbb{T}):=\{ G\in\mathbb{T}\;|\;\sum_{\alpha\in\mathbb{N}^d}G[\alpha]^2<\infty\}$ we have
\begin{align}
\mathcal{V}(\hat{X}, \underline{c},\overline{c})
&= \{v(\hat{x}) = \sum\limits_{\alpha\in\bbN^d} G[\alpha]\cdot P_\alpha(\hat{x})\;|\;  \nonumber
\\ &   G \in \ell^2(\mathbb{T}), \
   \underline{F}(G) \geq 0,\ \overline{F}(G) \leq 0
 \},\\
\label{eq:truncated_model_class}
\mathcal{M}(\underline{c},\overline{c})
&= \{ 
v(\hat{x}) = 
\sum\limits_{\alpha \in \Lambda} G[\alpha]\cdot P_{\alpha}(\hat{x})
\;|\; 
\nonumber
\\
&
G \in \mathbb{TT}_{\bm{r}},\ 
\underline{F}^{\bm{r}}(G) \geq 0,\
 \overline{F}^{\bm{r}}(G) \leq 0 \},
\end{align}
for constraint functions $\underline{F},\overline{F}\colon \ell^2(\mathbb{T}) \to \mathbb{R}$
and \\$\underline{F}^{\bm{r}},\overline{F}^{\bm{r}} \colon \ell^2(\mathbb{TT}_{\bm{r}})\to \mathbb{R}$ implicitly bounding the coefficient tensors.
Note that due to the orthonormality of $\{P_{\alpha}\}_{\alpha\in\bbN^d}$ in $\mathcal{V}(\hat{X})$ for every $v\in\mathcal{V}(\hat{X})$ it holds
\begin{equation}
	 \|v\|_{\mathcal{V}} = \| G \|_{\ell^2(\mathbb{T})}\quad\text{with}\quad  v = \sum\limits_{\alpha}G[\alpha]P_\alpha \in \mathcal{V}.
\end{equation}

Additionally, we define a \emph{loss function} $\iota\colon\mathcal{V}(\hat{X}, \underline{c}, \overline{c})\times \hat{X}\to\bbR$ such that $\iota(\cdot, \hat{x})$ is continuous for almost all $\hat x\in \hat{X}$ and $\iota(v, \cdot)$ is integrable with respect to the weight function $w$ of $\mcV(\hat{X})$ for every $v\in \mcV(\hat{X}, \underline{c},\overline{c})$.
Then, we consider the \emph{cost functional} $\mathscr{J}\colon\mcV(\hat{X}, \underline{c}, \overline{c})\to \bbR$ given by
\begin{equation}
  \label{eq:lossfunctional}
  \mathscr{J}(v) := \int_{\hat X} \iota(v, \hat x) w(\hat x)\mathrm{d}\lambda(\hat x).
\end{equation}

To further analyze the approximability in the given TT format using sampling techniques, we define two common discrepancy measures for probability density functions.
\begin{lemma}(KL loss compatibility)
\label{lem:KL_comp}
Let $h^\ast\in\mcV(\hat{X}, 0, c^*)$ for $c^* <\infty$ and $0<\underline{c} < \overline{c}<\infty$.  
Then 
\begin{align}
\label{eq:definition_iota_kl}
\mcV(\hat{X}, \underline{c}, \overline{c})\ni g \mapsto \iota(g, \hat{x}) &\phantom{:}= \iota(g, \hat{x}, h^\ast) \nonumber\\&:= -\log(g(x))h^\ast(x)
\end{align}
is uniformly bounded and Lipschitz continuous on\\ $\mathcal{M}_{\bm{r}}(\underline{c},\overline{c})$
if $P_{\alpha} \in L^\infty(\hat{X})$ for every $\alpha\in\Lambda$. Furthermore, $\mathscr{J}$ is globally Lipschitz continuous on the metric space $(\mcV(\hat{X}, \underline{c}, \overline{c}), d_{\mathcal{V}(\hat{X},\underline{c},\overline{c})})$.
\end{lemma}
\begin{proof}
The loss $\iota$ is bounded on $\mathcal{M}_{\bm{r}}(\underline{c},\overline{c})$ since $0 < \underline{c} < \overline{c}<\infty$. 
Let $g_1, g_2\in \mathcal{V}_{\bm{r}}(\hat{X},\underline{c},\overline{c})$ with coefficient tensors $G_1$ and $G_2\in\mathbb{TT}_{\bs{r}}$, then
\begin{equation}
	\label{eq:KL_comp_L}
	\abs{\iota(g_1, \hat{x}) - \iota(g_2, \hat{x})} 
	\leq
	\underbrace{\frac{1}{\underline{c}}\sup\limits_{\hat{x} \in \hat{X}} \{h^\ast(\hat{x})\}}_{:=C^\ast<\infty} |g_1(\hat{x})- g_2(\hat{x})|.
\end{equation}
The global Lipschitz continuity of $\mathscr{J}$ follows by using~\eqref{eq:KL_comp_L} and
\begin{align}
|\mathscr{J}(g_1)-\mathscr{J}(g_2)| &\leq C^\ast \|g_1-g_2\|_{L^1(\hat{X},w)}\nonumber \\& \leq C C^\ast d_{\mathcal{V}(\hat{X},\underline{c},\overline{c})}(g_1,g_2),
\end{align}
with a constant $C$ related to the embedding of $L^2(\hat{X},w)$ into $L^1(\hat{X},w)$.
If $g_1,g_2$ are in $\mathcal{M}_{\bm{r}}(\underline{c},\overline{c})$ then by Parseval's identity and the finite dimensionality of $\mathcal{M}_{\bm{r}}(\underline{c}, \overline{c})$ there exists $c=c\left(\sup_{\alpha\in\Lambda} \|P_\alpha\|_{L^\infty(\hat{X})}\right)>0 $ such that
\begin{align}
\label{eq_KL_comp_lifting}
|g_1(x)-g_2(x)| \leq c\|G_1-G_2\|_{\ell^2(\mathbb{T})} &= c\|g_1-g_2\|_{\mathcal{V}} \nonumber\\&= c\, d_{\mathcal{V}(\hat{X},\underline{c},\overline{c})}(g_1,g_2),
\end{align}
which yields the Lipschitz continuity on $\mathcal{M}_{\bm{r}}(\underline{c}, \overline{c})$.
Now let $g_1,g_2 \in\mathcal{V}(\hat{X},\underline{c},\overline{c})$.
The global Lipschitz continuity of $\mathscr{J}$ follows by using~\eqref{eq:KL_comp_L} and
\begin{align}
|\mathscr{J}(g_1)-\mathscr{J}(g_2)| &\leq C^\ast \|g_1-g_2\|_{L^1(\hat{X},w)}\nonumber\\& \leq C C^\ast d_{\mathcal{V}(\hat{X},\underline{c},\overline{c})}(g_1,g_2),
\end{align}
with a constant $C$ related to the embedding of $L^2(\hat{X},w)$ into $L^1(\hat{X},w)$.
\end{proof}

\begin{lemma}($L^2$-loss compatibility)
\label{lem:quadratic_loss}
Let $h^\ast\in\mcV(\hat{X}, 0, \overline{c})$ for $\overline{c}<\infty$.
Then
\begin{equation}
\label{eq:definition_iota_l2}
\mcV(\hat{X}, 0, \overline{c})\ni g \mapsto \iota(g, \hat{x}) = \iota(g, \hat{x}, h^*) := |g(\hat{x})-h^*(\hat{x})|^2
\end{equation}
is uniformly bounded and Lipschitz continuous on\\ $\mathcal{M}_{\bm{r}}(0,\overline{c})$
 provided $P_{\alpha} \in L^\infty(\hat{X})$ for every $\alpha\in\Lambda$. 
\end{lemma}
\begin{proof}
Let $g_1, g_2\in \mathcal{V}(\hat{X},0,\overline{c})$. Then
\begin{align}
	\abs{\iota(g_1, \hat{x}) - \iota(g_2, \hat{x})} &\leq 
	| g_1(\hat{x}) - g_2(\hat{x})|\cdot|g_2(\hat{x})+g_2(\hat{x})|\nonumber\\&
	+ 2|g_1(\hat{x})- g_2(\hat{x})| h^\ast(\hat{x}).
\end{align}
Due to $\overline{c}<\infty$ the Lipschitz property follows as in the proof of Lemma~\ref{lem:KL_comp} if $g_1,g_2$ in $\mathcal{M}_{\bm{r}}(\underline{c},\overline{c})$. 
\end{proof}

To examine the VMC convergence in our setting, we recall the analysis of~\cite{ESTW19} in a slightly more general manner. 
The target objective of the method is to find a minimizer 
\begin{equation}
\label{eq:min_problem_cont}
v^\ast \in \argmin\limits_{v\in\mathcal{V}(\hat{X},\underline{c},\overline{c})}\mathscr{J}(v).
\end{equation}
Due to the infinite dimensional setting we confine the minimization problem in~\eqref{eq:min_problem_cont}
to our model class $\mathcal{M}=\mathcal{M}_{\bm{r}}(\underline{c},\overline{c})$.
This yields the minimization problem
\begin{equation}
\label{eq:min_problem_M}
\text{find} \; v^\ast_{\mathcal{M}} \in \argmin\limits_{v\in\mathcal{M}}\mathscr{J}(v).
\end{equation}
A crucial step is then to consider the empirical functional instead of the integral in $\mathscr{J}$, namely
\begin{equation}
\label{eq:empirical_cost_functional}
\mathscr{J}_N(v) := \frac{1}{N}\sum\limits_{k=1}^N \iota(v;\hat{x}^k),
\end{equation}
with independent samples $\{\hat{x}^k\}_{k\leq N}$ distributed according to the measure $w\lambda$ with a (possibly rescaled) weight function $w$ with respect to the Lebesgue measure $\lambda$. 
The corresponding empirical optimization problem then takes the form
\begin{equation}
\label{eq:min_problem_MN}
\text{find} \; v^\ast_{\mathcal{M},N} \in \argmin\limits_{v\in\mathcal{M}}\mathscr{J}_N(v).
\end{equation}
The analysis examines different errors with respect to $h^\ast\in\mathcal{V}(\hat{X},0,\overline{c})$ defined by
\begin{align}
&&\mathcal{E} &:= \left|\mathscr{J}(h^\ast) - \mathscr{J}\left(v^\ast_{\mathcal{M},N}\right)\right|,&&\\\ 
,&&\mathcal{E}_{\mathrm{app}} &:= \left| \mathscr{J}(h^\ast) - \mathscr{J}\left(v^\ast_{\mathcal{M}}\right)\right| 
,&&\\ 
&&\mathcal{E}_{\mathrm{gen}} &:= \left|\mathscr{J}\left(v^\ast_{\mathcal{M}}\right) - \mathscr{J}\left(v^\ast_{\mathcal{M},N}\right)\right|,&& 
\end{align}
denoting the VMC-, approximation- and generalization error respectively.
By a simple splitting, the VMC error can be bounded by the approximation and the generalization error, namely
\begin{equation}
\label{eq:vmc_error_bound_app_gen}
\mathcal{E} \leq \mathcal{E}_{\mathrm{app}} + \mathcal{E}_{\mathrm{gen}}.
\end{equation}
Due to the global Lipschitz property on $\mathcal{V}(\hat{X},\underline{c},\overline{c})$ with $\underline{c} > 0 $ in the setting of~\eqref{eq:definition_iota_kl}  or  $\underline{c}\geq 0$ as in~\eqref{eq:definition_iota_l2}, the approximation error can be bounded by the best approximation in $\mathcal{M}$.
In particular there exists $C>0$ such that
\begin{equation}
\label{eq:quasi best approximation error}
\mathcal{E}_{\mathrm{app}} \leq C \inf\limits_{v\in \mathcal{M}}\|h^\ast-v\|_{\mathcal{V}(\hat{X})}^2.
\end{equation} 

We note that such an estimation by the best approximation in $\mathcal{M}$ with respect to the $\mathcal{V}(\hat{X})$-norm may not be required when using the Kullback-Leibler divergence if one is interested directly in the best approximation in this divergence.
Then the assumption $\underline{c}>0$ can be relaxed in the construction of $\mathcal{V}(\hat{X},\underline{c},\overline{c})$ since no global Lipschitz continuity of $\mathscr{J}$ in Lemma~\ref{lem:KL_comp} is required.
Thus the more natural subspace of $\mathcal{V}(\hat{X},0,\overline{c})$ of absolutely continuous functions with respect to $h^\ast$ may be considered instead.

It remains to bound the statistical generalization error $\mathcal{E}_{\mathrm{gen}}$.
For this the notion of covering numbers is required. Let $(\Omega,\mathcal{F},\mathbb{P})$ be an abstract probability space.
\begin{definition}(\textbf{covering number})
\label{def:covering_number}
Let $\epsilon > 0$. The covering number $\nu(\mathcal{M},\epsilon)$ denotes the minimal number of open balls of radius $\epsilon$ with respect to the metric $d_{\mathcal{V}(\hat{X},\underline{c},\overline{c})}$ needed to cover $\mathcal{M}$.
\end{definition}

\begin{lemma}
\label{lem:gen_error}
Let $\iota$ be defined as in~\eqref{eq:definition_iota_kl} or~\eqref{eq:definition_iota_l2}. 
Then there exist $C_1,C_2>0$ only depending on the uniform bound and the Lipschitz constant of $\mathcal{M}$ given in Lemma~\ref{lem:KL_comp} and~\ref{lem:quadratic_loss}, respectively, such that for $\epsilon >0$ and $N\in\mathbb{N}$ denoting the number of samples in the empirical cost functional in~\eqref{eq:empirical_cost_functional} it holds 
\begin{equation}
\mathbb{P}[\mathcal{E}_{\mathrm{gen}}>\epsilon]  \leq 2\nu(\mathcal{M}, C_2^{-1}\epsilon) \delta(1/4\epsilon, N),
\end{equation}
with $\delta(\epsilon,N)\leq 2\exp(-2\epsilon^2N/C_1^2)$.
\end{lemma}
\begin{proof}
The claim follows immediately from Lemmas~\ref{lem:KL_comp} and~\ref{lem:quadratic_loss}, respectively, and~\cite[Thm. 4.12, Cor. 4.19]{ESTW19}.
\end{proof}

\begin{remark}[choice of $\underline{c},\overline{c}$ and $\hat{X}$]
Due to the layer based representation in~\eqref{eq:density_decomposition} and~\eqref{eq:layer representation + split} on each layer $\hat{X}^\ell = \Phi^{-1}(X^\ell)$
we have the freedom to choose $\underline{c}$ separately.
In particular, assuming that the perturbed prior $\tilde{f}_0$ decays per layer, we can choose $\underline{c}$ according to the decay and with this control the constant in~\eqref{eq:KL_comp_L}.	
\end{remark}

\section{Error estimates}
\label{sec:error estimates}
This section is devoted to the derivation of a priori error estimates for the previously introduced construction in terms of the Hellinger distance and Kullback-Leibler divergence.
We employ the VMC approach from Section~\ref{section:tensor_train_regression} to the density layer approximation which leads to a convergence result.

Recall that our goal is to approximate the perturbed prior $\tilde{f}_0$ given some transport $\tilde{T}$ represented by a function $\tilde{f}_0^{\mathrm{Trun},\mathrm{TT}}$ defined by
\begin{equation}
    \label{eq:full_discrete_representation}
    \tilde{f}_0^{\on{Trun},\mathrm{TT}}(x) :=
    C_L^{\mathrm{TT}}
    \left\{
    \begin{array}{ll}
    \tilde{f}_0^{\ell,\mathrm{TT}}(x), & x\in X^\ell, \ell = 1,\ldots,L, \\
    f_{\Sigma,\mu}(x), & x \in X^{L+1}.
    \end{array}
    \right.
\end{equation}
Here, $C_L^{\mathrm{TT}}:=(C_L^{<}+ C_L^{>,\mathrm{TT}})^{-1}$ with $C_L^{<}$ from \eqref{eq:norm constant<} and 
\begin{equation}
    C_L^{>,\mathrm{TT}} := \sum\limits_{\ell=1}^L \int\limits_{X^\ell} \tilde{f}_0^{\ell,\mathrm{TT}}(x)\,\mathrm{d}\lambda(x).
\end{equation}
Furthermore, $\tilde{f}_0^{\ell,\mathrm{TT}} = \hat{f}_0^{\ell,\mathrm{TT},N_\ell}\circ\left(\Phi^{\ell}\right)^{-1}$ is the pullback of a function $\hat{f}_0^{\ell,\mathrm{TT},N_\ell}$ in $\mathcal{M}^\ell=\mathcal{M}(\underline{c}_\ell,\overline{c}_\ell)$ over $\hat{X}^\ell$.
Analog to the empirical minimization problem~\eqref{eq:empirical_cost_functional} with $w_\ell = |\operatorname{det}\mathcal{J}_{\Phi^\ell}|$, we choose $\hat{f}_0^{\ell,\mathrm{TT},N_\ell}$ as 
\begin{equation}
\label{eq:layer_minimization}
\hat{f}_0^{\ell,\mathrm{TT},N_\ell} \in\argmin\limits_{v\in\mathcal{M}^{\ell}} \frac{1}{N_\ell}\sum\limits_{k=1}^{N_\ell} \iota(v,\hat{x}^k,\hat{f}_0),
\end{equation}
with samples $\{\hat{x}^k\}_{k=1}^{N_\ell}$ drawn from the (possibly rescaled) finite measure $w_\ell\lambda$.
The connection to the actual approximation of the target density $f$ given by 
\begin{equation}
\label{eq:posterior_apprx}
\tilde{f}^{TT}:=  \tilde{f}_0^{\mathrm{Trun},\mathrm{TT}} \circ \tilde{T}^{-1} \otimes |\mathcal{J}_{\tilde{T}^{-1}}|
\end{equation}
is reviewed in the following.
We refer to Figure~\ref{fig:overview} for a visual presentation of the involved objects, approximations and transformations.
\begin{figure*}
  \centering
\begin{tikzpicture}
\foreach \x in {0,3,6}{
\draw[ fill = blue, opacity = 0.25] (0.42, -7+\x) -- (0.85, -7+\x) -- (0.85, -5+\x) -- (0.42,-5+\x) -- cycle;
\fill[blue,even odd rule, opacity = 0.25] (5.25,-\x) circle (0.01*30) (5.25,-\x) circle (0.01*55);
}

\foreach[evaluate=\c using int(\l-5)] \l in  {5,20,35,50,65,80,95}{
  \def\opac{1-0.008*\l}
  \draw[lightgray!\c!black, thick, opacity = \opac, rotate around={-45:(10,-1.)}] (10,-1.) ellipse ({0.004*\l} and {0.015*\l});
}


\foreach[evaluate=\c using int(\l-5)] \l in {5,20,35,50,65,80,95}{ 
  \def\opac{1-0.008*\l}
  \draw[lightgray!\c!black, thick, opacity = \opac] (5.25,0) circle(0.01*\l);
}

\foreach[evaluate=\c using int(\l-5)] \l in {5,20,35,50,65,80,95,110,125}{ 
  \def\opac{1-0.005*\l}
  \draw[lightgray!\c!black, thick, opacity = \opac] (0.01*\l,-0.8)--(0.01*\l,0.8);
}

\foreach[evaluate=\c using int(\l-5)] \l in  {5,20,35,50,65,80}{
  \def\opac{1-0.008*\l}
  \draw[lightgray!\c!black, thick, opacity = \opac, rotate around={-45:(10,-4)}] (10,-4) ellipse ({0.004*\l} and {0.016*\l});
}
\foreach[evaluate=\c using int(\l-5)] \l in  {95,105}{
  \def\opac{1-0.008*\l}
  \draw[lightgray!\c!black, thick, opacity = \opac, rotate around={-45:(10,-4)}] (10,-4) ellipse ({0.005*\l} and {0.014*\l});
}

\foreach[evaluate=\c using int(\l-5)] \l in {5,20,35,50,65,80,95}{ 
  \def\opac{1-0.008*\l}
  \draw[lightgray!\c!black, thick, opacity = \opac, rotate around={45:(5.25,-3)}] (5.25,-3) ellipse({0.012*\l} and {0.01*\l});
}

\foreach[evaluate=\c using int(\l-5)] \l in {5,20,35,50,65,80,95, 110, 125}{ 
  \def\opac{1-0.005*\l}
  \draw[lightgray!\c!black, thick, opacity = \opac] (0.01*\l,-0.8 -3)to [out=90-0.6*\l,in=-90](0.01*\l,0.8-3);
}

\foreach[evaluate=\c using int(\l-5)] \l in {5,20,35,50,65}{ 
  \def\opac{1-0.005*\l}
  \draw[lightgray!\c!black, thick, opacity = \opac] (0.01*\l,-0.8 -6)to [out=90-0.6*\l,in=-90](0.01*\l,0.8-6);
}
\foreach[evaluate=\c using int(\l-5)] \l in {80, 95, 110, 125}{ 
  \def\opac{1-0.005*\l}
  \draw[lightgray!\c!black, thick, opacity = \opac] (0.1+0.01*\l,-0.8 -6)to [out=90,in=-90](0.1+0.01*\l,0.8-6);
}

\foreach[evaluate=\c using int(\l-5)] \l in {5,20,35,50,65}{ 
  \def\opac{1-0.008*\l}
  \draw[lightgray!\c!black, thick, opacity = \opac, rotate around={45:(5.25,-6)}] (5.25,-6) ellipse({0.012*\l} and {0.01*\l});
}
\foreach[evaluate=\c using int(\l-5)] \l in {90,105}{ 
  \def\opac{1-0.008*\l}
  \draw[lightgray!\c!black, thick, opacity = \opac, rotate around={45:(5.25,-6)}] (5.25,-6) ellipse({0.01*\l} and {0.01*\l});
}

\node at (0.65, 1.3) {\scriptsize (approximation domain)};
\node at (5.25, 1.3) {\scriptsize (reference domain)};
\node at (10, 1.3) {\scriptsize (target domain)};

\node at (0.65, 1.75) {${\color{blue!60!white}\hat{X}^\ell}\subset\hat{X}$};
\node at (5.25, 1.75) {${\color{blue!60!white}X^\ell}\subset X$};
\node at (10.0, 1.75) {$Y$};

\draw[->] (1.75,0) -- (3.75,0);
\draw[->] (1.75,-3) -- (3.75,-3);
\draw[->] (1.75,-6) -- (3.75,-6);

\node at (2.75,0.5) {$\Phi^\ell$};
\node at (2.75,0.5-3) {$\Phi^\ell$};
\node at (2.75,0.5-6) {$\Phi^\ell$};

\draw[->] (6.5,0) to [out = 0, in = 180] (8.5,-1);
\draw[->] (6.5,-3) to [out = 0, in = 180] (8.5,-1.5);
\draw[->] (6.5,-3) to [out = 0, in = 180] (8.5,-1.5);
\draw[->] (6.5,-6) to [out = 0, in = 180] (8.5,-4);
\node at (7.5, 0) {$T$};
\node at (7.5, -3.5) {$\tilde{T}$};


\draw[->] (-0.5,-0.8 -2.25) to [out= 180, in = 180] (-0.5,-0.8 - 5.25);
\node at (-1.6, -0.8 - 3.75) {\rotatebox{90}{VMC (Section~\ref{section:tensor_train_regression})}};

\node at (0.8, -4.3 -3) {\scriptsize $\hat{f}_0^{\ell,\mathrm{TT},N_\ell}$ from~\eqref{eq:full_discrete_representation}};
\node at (0.7, -4.3 ) {\scriptsize$\hat{f}_0 = \hat{f}_0^\ell$ from~\eqref{eq:local_pert_prior}};

\node at (0.7, -4.3 + 3) {\scriptsize$f_0\circ \Phi^{\ell}$}; 

\node at (5 , -4.3 +3) {\scriptsize $f_0$ from~\eqref{eq:exact_prior}};
\node at (5 , -4.3 +0  
) {\scriptsize $\tilde{f}_0$ from~\eqref{eq:pert_prior_density}};
\node at (5 , -4.3 -3) {\scriptsize $\tilde{f}_0^{\on{Trun},\mathrm{TT}}$ from~\eqref{eq:full_discrete_representation}};

\node at (10 , -4.3 +2) {\scriptsize $f$ from~\eqref{eq:main density}};
\node at (10 , -4.4 -1  
) {\scriptsize $\tilde{f}^{\mathrm{TT}}$ from~\eqref{eq:posterior_apprx}};

\end{tikzpicture}
\caption{Overview of the presented method sketching the different involved transformations and approximations with references to the respective equations.}
\label{fig:overview}
\end{figure*}
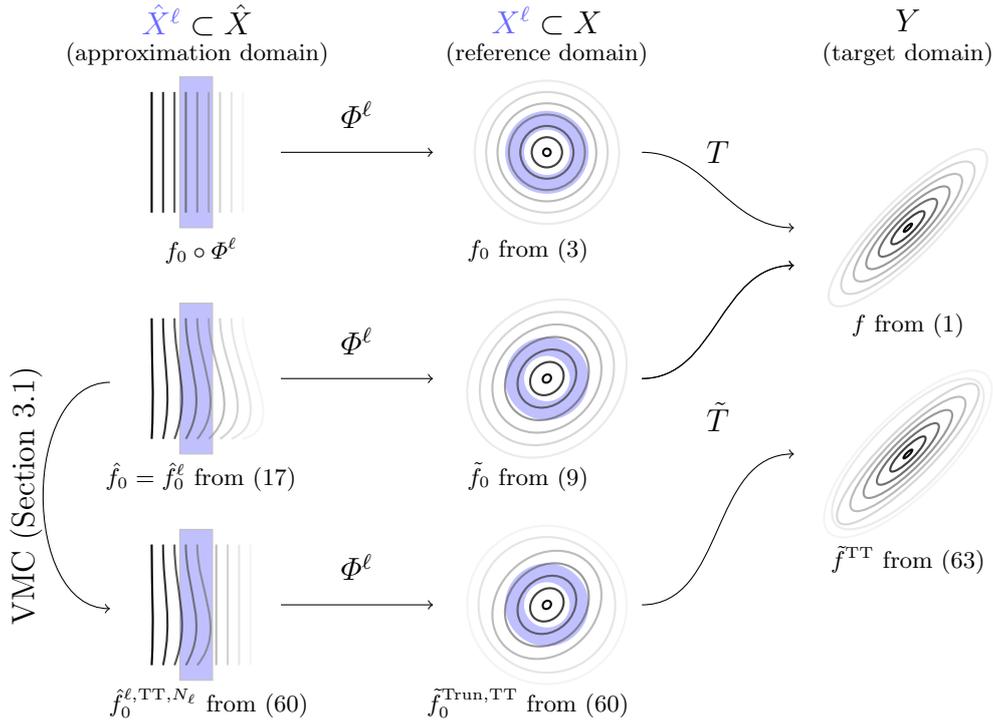

We first consider the relation of a target density $f$ and its perturbed prior $\tilde{f}_0$.
Since the transport $\tilde{T}$ maps $X$ to $Y$, an error functional $\on{d}(Y;\cdot,\cdot)$ has to satisfy 
\begin{equation}
\label{eq:error_equality}
\on{d}\left(Y; f, \tilde{f}^{TT}\right) = \on{d}\left(X; \tilde{f}_0, \tilde{f}_0^{\mathrm{Trun},\mathrm{TT}}\right).
\end{equation}
This property ensures that control of the error of the approximation in terms of the perturbed prior with respect to $\on{d}(X;\cdot,\cdot)$ transfers directly to $f$.
Note that this criterion is canonical as passing to the image space of some measurable function is fundamental in probability theory. 

Prominent measures of discrepancy for two absolutely continuous Lebesgue probability density functions $h_1$ and $h_2$ on some measurable space $Z$ are the Hellinger distance  
\begin{equation}
 \on{d}_{\mathrm{Hell}}(Z, h_1, h_2) = \int\limits_{Z} \left(\sqrt{h_1}(z)-\sqrt{h_2}(z)\right)^2\,\mathrm{d}\lambda(z),
\end{equation}
and the Kullback-Leibler divergence 
\begin{equation}
\label{eq:Def KL}
\on{d}_\mathrm{KL}(Z,h_1,h_2) = \int\limits_{Z} \log\left(\frac{h_1(z)}{h_2(z)}\right) h_1(z)\,\mathrm{d}\lambda(z).
\end{equation}
For the Hellinger distance, the absolute continuity assumption can be dropped from an analytical point of view.
Observe that both $\on{d}_{\mathrm{Hell}}$ and $\on{d}_\mathrm{KL}$ both satisfy~\eqref{eq:error_equality}. 
\begin{lemma}
\label{lem:dX_dY}
Let $\sharp\in\{\mathrm{Hell},\mathrm{KL}\}$, then it holds
\begin{equation}
\label{eq:sharp_f_f0tilde}
\on{d}_{\sharp}(Y; f, \tilde{f}^{TT} )
= \on{d}_{\sharp}(X; \tilde{f}_0, \tilde{f}_0^{\mathrm{Trun},\mathrm{TT}} ).
\end{equation}
\end{lemma}
\begin{proof}
	We only show~\eqref{eq:sharp_f_f0tilde} for $\sharp=\mathrm{KL}$ since $\sharp=\mathrm{Hell}$ follows by similar arguments. 
By definition
\begin{equation}
	\on{d}_{\mathrm{KL}}(Y; f, \tilde{f}^{\mathrm{TT}} )
	= \int\limits_{Y} \log\left(\frac{f(y)}{\tilde{f}^{\mathrm{TT}}(y)}\right)f(y)\,\mathrm{d}\lambda(y),
\end{equation}
and the introduction of the transport map $\tilde{T}$ yields the claim
\begin{align}
	&\phantom{=}\int\limits_{X} \log\left(
	\frac{f\circ \tilde{T}(x)}{\tilde{f}^{\mathrm{TT}}\circ \tilde{T}(x)}
	\cdot
	\frac{|\operatorname{det}\mathcal{J}_{\tilde{T}}(x)|}{|\operatorname{det}\mathcal{J}_{\tilde{T}}(x)|}
	\right)\tilde{f}_0(x)\,\mathrm{d}\lambda(x) \nonumber \\
	&= \on{d}_{\mathrm{KL}}(X; \tilde{f}_0, \tilde{f}_0^{\mathrm{Trun},\mathrm{TT}} ).
\end{align}
\end{proof}

With the previous results and notations, the following assumption turns out to be required for the convergence result.
\begin{assumption}
\label{ass:theorems}
For a target density $f\colon Y\to\mathbb{R}_+$ and a transport map $\tilde{T}\colon X \to Y$, there exists a simply connected compact domain $K$ such that $\tilde{f}_0=(f\circ T)\otimes|\operatorname{det}\mathcal{J}_T|\in L^2(K)$ 
 has outer polynomial exponential decay with polynomial $\pi^+$ on $X\setminus K$.
 Consider the symmetric positive definite matrix $\Sigma\in\bbR^{d, d}$ and $\mu\in\bbR^d$ as the covariance and mean for the outer approximation $f_{\Sigma, \mu}$.
Furthermore, let $K=\bigcup_{\ell=1}^L \overline{X^\ell}$ with $X^\ell$ being the image of a rank-1 stable diffeomorphism $\Phi^\ell\colon \hat{X}^\ell\to X^\ell$ for every $\ell = 1,\ldots,L$. 
\end{assumption}

We can now formulate the main theorem of this section regarding the convergence of the developed approximation with respect to the Hellinger distance and the KL divergence.

\begin{theorem} (\textbf{A priori convergence})
\label{thm:Hell}
Let Assumption~\ref{ass:theorems} hold and let a sequence of sample sizes $(N^\ell)_{\ell=1}^L\subset \mathbb{N}$ be given.
For every $\ell =1,\ldots,L$, consider bounds $0<\underline{c}^\ell<\overline{c}^\ell<\infty$ and let $\tilde{f}^{\mathrm{TT}}$ be defined as in~\eqref{eq:posterior_apprx}. 
Then there exist constants $C,C_\Sigma,C^\ell,C_\iota^\ell >0$, $\ell=1,\ldots,L$, such that for
$\sharp\in\{\mathrm{KL},\mathrm{Hell}\}$ 
\begin{align}
\on{d}_{\sharp}(Y,f,\tilde{f}^{\mathrm{TT}}) &\leq
 C\left(\sum\limits_{\ell=1}^L \left( \mathcal{E}_{\mathrm{best}}^\ell + \mathcal{E}_{\mathrm{sing}}^\ell + \mathcal{E}_{\mathrm{gen}}^\ell \right) + \mathcal{E}_{\mathrm{trun}}^\sharp\right).
 \label{eq:sharp-est}
\end{align}
Here, $\mathcal{E}_{\mathrm{best}}^\ell$ denotes the error of the best approximation $v_\Lambda^\ell$ to $\hat{f}_0^\ell$ in the full truncated polynomial space $\mathcal{V}_{\Lambda}^\ell(\underline{c}^\ell,\overline{c}^\ell) =\mathcal{V}_{\Lambda}^\ell\cap \mathcal{V}(\hat{X}^\ell,\underline{c}^\ell,\overline{c}^\ell)$ given by
\begin{equation*}
 \mathcal{E}_{\mathrm{best}}^\ell:=
\|\hat{f}_0^\ell - v_\Lambda^\ell\|_{\mathcal{V}(\hat{X}^\ell)}
=\inf\limits_{v^\ell\in \mathcal{V}_{\Lambda}^\ell(\underline{c}^\ell,\overline{c}^\ell)}
\| \hat{f}_0^\ell - v^\ell \|_{\mathcal{V}(\hat{X}^\ell)},
\end{equation*}
$\mathcal{E}_{\mathrm{sing}}^\ell$ is the low-rank approximation error of the algebraic tensor associated to $v_\Lambda^\ell$ and the truncation error $\mathcal{E}_{\mathrm{trun}}$ is given by
\begin{align*}
\left(\mathcal{E}_{\mathrm{trun}}^\mathrm{Hell}\right)^2 &:= 
     \norm{\exp{(-\pi^+)}}_{L^1(X\setminus K)} + \Gamma\left(d/2, C_\Sigma R^2\right),\\
\mathcal{E}_{\mathrm{trun}}^\mathrm{KL} &:= 
     \int\limits_{X\setminus K} \left(\frac{1}{2}\|x\|_{\Sigma^{-1}}^2 + \tilde{\pi}^+(x)\right)e^{-\tilde{\pi}^+(x)}\,\mathrm{d}\lambda(x).
\end{align*}
Furthermore, for any $(\epsilon^\ell)_{\ell=1}^L \subset \mathbb{R}_+$ the generalization errors $\mathcal{E}_{\mathrm{gen}}^\ell$ can be bounded in probability
\begin{equation*}
\mathbb{P}(\mathcal{E}_{\mathrm{gen}}^\ell > \epsilon^\ell) \leq 2\nu(\mathcal{M}^\ell, C^\ell \epsilon^\ell)\delta^\ell(1/4\epsilon^\ell, N^\ell)
\end{equation*}
with $\nu$ denoting the covering number from Definition~\ref{def:covering_number} and 
$\delta^\ell(\epsilon,N)\leq 2\exp(-2\epsilon^2N/{C_\iota^\ell})$.
\end{theorem}

\begin{proof}
	We first prove~\eqref{eq:sharp-est} for $\sharp=\mathrm{Hell}$ and point out that the Hellinger distance can be bounded by the $L^2$ norm.  
	Note that $|\sqrt{a}-\sqrt{b}| \leq \sqrt{|a-b|}$ for $a,b\geq 0$ and with Lemma~\ref{lem:dX_dY} it holds
	\begin{align*}
	\on{d}_{\mathrm{Hell}}(Y; f,\tilde{f}^{\mathrm{TT}})&=
	\on{d}_{\mathrm{Hell}}(X; \tilde{f}_0, \tilde{f}_0^{\mathrm{Trun},\mathrm{TT}} ) \\
	&\leq
	\|\tilde{f}_0 -\tilde{f}_0^{\mathrm{Trun},\mathrm{TT}}\|_{L^1(K)} \\
        &\quad
	+\|\tilde{f}_0 -\tilde{f}_0^{\mathrm{Trun},\mathrm{TT}}\|_{L^1(X\setminus K)}.
	\end{align*}
	Since $K = \cup_{\ell=1}^L X^\ell$ and $X^\ell$ are bounded, there exist constants $C(X^\ell)>0$, $\ell=1,\ldots,L$, such that
	\begin{align*}
	\|\tilde{f}_0 -\tilde{f}_0^{\mathrm{Trun},\mathrm{TT}}\|_{L^1(K)} &=
	\sum\limits_{\ell=1}^L\|\tilde{f}_0 -\tilde{f}_0^{\mathrm{Trun},\mathrm{TT}}\|_{L^1(X^\ell)}\\
	&\leq
	\sum\limits_{\ell=1}^L  C(X_\ell)
	\|\tilde{f}_0 -\tilde{f}_0^{\mathrm{Trun},\mathrm{TT}}\|_{L^2(X^\ell)}.
	\end{align*}
	Moreover, by construction
	\begin{equation}
	\label{eq:proof_Xell_VhatXell}
	\|\tilde{f}_0 -\tilde{f}_0^{\mathrm{Trun},\mathrm{TT}}\|_{L^2(X^\ell)}
	= 
	\|\hat{f}_0^\ell -\hat{f}_0^{\ell,\mathrm{TT},N_\ell} \|_{\mathcal{V}(\hat{X}^\ell)}.
	\end{equation}
	The claim follows by application of Lemmas~\ref{proposition:truncation},~\ref{lem:HOSVD} and~\ref{lem:gen_error} together with~\eqref{eq:vmc_error_bound_app_gen}. 
	
	To show~\eqref{eq:sharp-est} for $\sharp=\mathrm{Hell}$, note that by Lemma~\ref{lem:dX_dY} and the construction~\eqref{eq:full_discrete_representation} it holds
	\begin{align}
	\on{d}_{\mathrm{KL}}(Y; f,\tilde{f}^{\mathrm{TT}}) &= \sum_{\ell=1}^L \int_{X^\ell}\log\frac{\tilde{f}_0}{\tilde{f}_0^{\ell, \mathrm{TT}}} \tilde{f}_0\mathrm{d}\lambda(x) \nonumber \\
&\quad + \int_{X\setminus K} \log\frac{\tilde{f}_0}{f_{\Sigma, \mu}} \tilde{f}_0\mathrm{d}\lambda(x).
	\end{align}
	Using Lemma~\ref{proposition:truncation} we can bound the integral over $X\setminus K$ by the truncation error $\mathcal{E}_{\mathrm{trun}}$.
	Employing the loss function and cost functional of Lemma~\ref{lem:KL_comp} yields
	\begin{equation}
	\int_{X^\ell}\log\frac{\tilde{f}_0}{\tilde{f}_0^{\ell, \mathrm{TT}}} \tilde{f}_0\mathrm{d}\lambda(x) \leq \mathcal{E}_{\mathrm{app}}^\ell + \mathcal{E}_{\mathrm{gen}}^\ell.
	\end{equation}
	The claim follows by application of Lemmas~\ref{lem:HOSVD} and~\ref{lem:gen_error}
	together with~\eqref{eq:vmc_error_bound_app_gen}. 
\end{proof}

\subsection{Polynomial approximation in weighted $L^2$ spaces}
\label{sec:polynomial best approximation}
In order to make the error bound~\eqref{eq:sharp-est} in Theorem~\ref{thm:Hell} more explicit with respect to $\mathcal{E}_\textrm{best}$, we consider the case of a smooth density function with analytic extension.
The analysis follows the presentation in~\cite{babuvska2010stochastic} and leads to exponential convergence rates by an iterative interpolation argument based on univariate best approximation bounds by interpolation.
An analogous analysis for more general regularity classes is possible but not in the scope of this article.

Let $\hat{X} = \bigotimes_{i=1}^d \hat{X}_i\subset\mathbb{R}^d$ be bounded and $w=\otimes_{i=1}^d w_i\in L^\infty(\hat{X})$ a non-negative weight such that $\mathcal{C}(\hat{X})\subset\mathcal{V}:=L^2(\hat{X},w) = \bigotimes_{i=1}^d L^2(\hat{X}_i,w_i)$.

For a Hilbert space $H$, a bounded set $I\subset\mathbb{R}$ and a function $f\in \mathcal{C}(I;H)\subset L^2(I,w;H)$ with weight $w\colon I\to\mathbb{R}$, let $\mathcal{I}_n \colon \mathcal{C}(I;H) \to L^2(I,w;H)$ defined as
\begin{align*}
\mathcal{I}_n f(\cdot) = \sum\limits_{k=1}^{n+1} f(\hat{x}_k) \ell_k(\cdot),
\end{align*}
denote the continuous Lagrange interpolation operator.
The $\ell_k$ are polynomials of degree $k$ orthogonal in $L^2(I,w)$ and $(\hat{x}_k)_{k=1}^n$ are the roots, respectively.

Assume that $f\in \mathcal{C}(I;H)$ admits an analytic extension in the region of the complex plane $\Sigma(I;\tau) := \{z\in\mathbb{C} | \operatorname{dist}(z,I)\leq \tau\}$ for some $\tau > 0$.
Then, referring to~\cite{babuvska2010stochastic},
\begin{equation}
\|f - \mathcal{I}_nf\|_{L^2(I,w;H)} \lesssim \sigma(n,\tau) \max\limits_{z\in\Sigma(I;\tau)}\|f(z)\|_H,
\label{eq:analytic_decay}
\end{equation}
with $ \sigma(n,\tau):=2(\rho-1)^{-1} \exp{(-n\log(\rho))}$ and $\rho := 2\tau/|I| + \sqrt{1 + 4\tau^2/|I|^2}>1$.
By using an iterative argument over $d$ dimensions, a convergence rate for the interpolation of $f\in\mathcal{C}(\hat{X};\mathbb{R})\subset L^2(\hat{X},w;\mathbb{R})$ can be derived from the one dimensional convergence.
More specifically, let $\mathcal{I}_\Lambda:\mathcal{C}(\hat{X})\mapsto L^2(\hat{X},w)$ denote the continuous interpolation operator written as composition of a $1$-dimensional and a $d-1$-dimensional interpolation $\mathcal{I}_\Lambda := \mathcal{I}_{n_1}^1\circ\mathcal{I}_{n_2:n_d}^{2:d}$ with continuous $\mathcal{I}_{n_1}^1\colon \mathcal{C}(\hat{X}_1)\to L^2(\bigtimes_{i=2}^d\hat{X}_i, \otimes_{i=2}^d w_i)$ and $\mathcal{I}_{n_2,\ldots,n_d}^{2,\ldots,d}\colon \mathcal{C}(\bigtimes_{i=2}^d \hat{X}_i)\to H$ with $ H= L^2(\bigtimes_{i=2}^d\hat{X}_i,\otimes_{i=2}^dw_i)$.
Then, for $f\in\mathcal{C}(\hat{X})$ and some $C>0$ it follows
\begin{align*}
 \| f - \mathcal{I}_\Lambda f\| &\leq \| f - \mathcal{I}_{n_1}^1 f\| + \| \mathcal{I}_{n_1}^1(f-\mathcal{I}_{n_2,\ldots,n_d}^{2,\ldots,d}f)\|\\
&\lesssim
\| f - \mathcal{I}_{n_1}^1 f\| + \\
&
\sup\limits_{\hat{x_1}\in \hat{X}_1}\| f(x_1)-\mathcal{I}_{n_2,\ldots,n_d}^{2,\ldots,d}f(x_1)\|_{ H}. 
\end{align*}
The second term of the last bound is a $d-1$-dimensional interpolation and can hence be bounded uniformly over $\hat{x}_1$ by a similar iterative argument.
We summarize the convergence result for $\mathcal{E}_{\mathrm{best}}^\ell$ in the spirit of~\cite[Theorem 4.1]{babuvska2010stochastic}.

\begin{lemma}
\label{lem:decay_rate}
Let $\hat{f}\in \mathcal{C}(\hat{X}^\ell)\subset L^2(\hat{X}^\ell,w)$ admit an analytic extension in the region $$\Sigma(\hat{X}^\ell, (\tau_i^\ell)_{i=1}^d) = \bigtimes_{i=1}^d \Sigma(\hat{X}_i^\ell,\tau_i^\ell)$$ for some $\tau_i^\ell>0$, $\ell=1,\ldots,L$, $i=1,\ldots,d$.
Then, with $\sigma$ from~\eqref{eq:analytic_decay},
$$
\inf\limits_{v\in \mathcal{V}_{\Lambda}} \| \hat{f} - v\|_{L^2(\hat{X}^\ell,w)} \lesssim \sum\limits_{i=1}^d \sigma(n_i,\tau_i).
$$
\end{lemma}
In case that $\underline{c}\leq f(\hat{x}), v^\ast(\hat{x}) \leq \overline{c}$ is satisfied for $v^\ast:= \argmin_{v\in\mathcal{V}_\Lambda}\|f - v\|_{L^2(\hat{X}^\ell,w)}$, the decay rate carries over onto the space $\mathcal{V}_\Lambda^\ell(\underline{c}^\ell,\overline{c}^\ell)$.
If only $\underline{c}\leq f(\hat{x})\leq \overline{c}$ holds, the image of $v^\ast$ can be restricted to $[\underline{c},\overline{c}]$, see e.g.~\cite{optimalweighted17}.
This approximation in fact admits a smaller error than $v^\ast$.

\begin{remark}
The interpolation argument on polynomial discrete spaces could be expanded to other orthonormal systems such as trigonometric polynomial, admitting well-known Lebesque constants as in~\cite{da2013lebesgue}.
\end{remark}

\begin{remark}
Explicit best approximation bounds for appropriate smooth weights $w$, as in the case of spherical coordinates, can be obtained using partial integration techniques as in~\cite{mead1973convergence}.
There the regularity class of $f$ is based on high-order weighted Sobolev spaces based on derivatives of $w$ as in the case of classical polynomials.
\end{remark}

\section{Algorithm}
\label{sec:Algorithm}
Since a variety of techniques are employed in the density discretization, this section provides an exemplary algorithmic workflow to illustrate the required steps in practical applications (see also Figure~\ref{fig:affine density transport} for a sketch of the components of the method).
The general method to obtain a representation of the density~\eqref{eq:main density} by its auxiliary reference~\eqref{eq:pert_prior_density} is summarized in Algorithm~\ref{alg:algo_levelrecon}.
Based on this, the computation of possible quantities of interest such as moments~\eqref{eq:pert_pi_0} or marginals are considered in Sections~\ref{sec:moments} and~\ref{sec:marginals}, respectively.
In the following we briefly describe the involved algorithmic procedures.

\paragraph{\textbf{Computing the transformation}}
Obtaining a suitable transport map is a current research topic and examined \eg in~\cite{papamakarios2019normalizing,parno2018transport,tran2019discrete,marzouk2016introduction}.
In Section~\ref{sec:transportMaps}, two naive options are introduced.
In the numerical applications, we employ an affine transport and also illustrate the capabilities of a quadratic transport in a two-dimensional example.
For the affine linear transport we utilize a semi-Newton optimizer to obtain the maximum value of $f$ and an approximation of the Hessian at the optimal value, see Section~\ref{sec:affine transport}. 
For the construction of a quadratic transport we rely on the library \verb#TransportMaps# \cite{TM}.
We summarize the task to provide the (possibly inexact) transport map in the function 
\begin{equation}
  \label{eq:ComputeTrafo}
  \tilde{T} \leftarrow \ComputeTrafo[f].
\end{equation}

In the following paragraphs we assume $\Phi^\ell$ to be the multivariate polar transformation as in Example~\ref{ex:polar_coords}, defined on the corresponding hyperspherical shells $\hat{X}^\ell$.
We refer to $\hat{X}^\ell_1$ as the \emph{radial dimension} and $\hat{X}^\ell_{i}$ as the \emph{angular dimensions} for $1 < i \leq d$. 
The computations on each shell $\hat{X}^\ell, \ell=1,\ldots,L$ are fully decoupled and suitable for parallelization. Note that the proposed method is easily adapted to other transformations $\Phi^\ell$.

\paragraph{\textbf{Generating an orthonormal basis}}
To obtain suitable finite dimensional subspaces, one has to introduce spanning sets that allow for an efficient computation of \eg moments~\eqref{eq:transport_moment_equation} and the optimization of the functional~\eqref{eq:lossfunctional}.
Given a fixed dimension vector $\bs{n}^\ell\in\mathbb{N}^d$ for the current $\hat{X}^\ell$, $\ell=1, \ldots, L$, and by the chosen parametrization via $\Phi^\ell$ introducing the weight $w^\ell$, the function
\begin{equation}
  \label{eq:GenerateOnb}
  \mathcal{P}^\ell = \{\mathcal{P}_i^\ell\}_{i=1}^{d} \leftarrow \GenerateONB[\hat{X}^\ell, \bs{n}^\ell, w^\ell, \tau_\on{GS}]
\end{equation}
can be split into three distinct algorithmic parts as follows.

\begin{itemize}
  \item \textit{1st coordinate $\hat{x}_1$}: The computation of an orthonormal polynomial basis $\{P_{1,\alpha}^\ell\}_{\alpha}$ with respect to the weight $w^\ell_1(\hat{x}_1) = \hat{x}_1^{d-1}$ in the radial dimension by a stabilized Gram-Schmidt method.
  This is numerically unstable since the involved summations cause cancellation. 
  As a remedy, we define \emph{arbitrary precision polynomials} with a significant digit length $\tau_{\mathrm{mant}}$ to represent polynomial coefficients.
  By this, point evaluations of the orthonormal polynomials and computations of integrals of the form

\begin{equation}
    \label{eq:arbitraryprecintegral}
    \int_{\hat{X}^\ell_1} \hat{x}_1^{m} P^\ell_{1, \alpha}(\hat{x}_1) \hat{x}_1^{d-1}\mathrm{d}\lambda(\hat{x}_1), \quad m\in\mathbb{N},
\end{equation}
\eg required for computing moments with polynomial transport, can be realized with high precision. 
  The length $\tau_{\mathrm{mant}}$ is set to $100$ in the numerical examples and the additional run-time is negligible as the respective calculations can be precomputed.
  \item \textit{2nd coordinate $\hat{x}_2$}: Since $\hat{X}_2^\ell=[0, 2\pi]$ and to preserve periodicity, we employ trigonometric polynomials given by
  \begin{equation}
    P^\ell_{2, j}(\hat{x}_2) = \left\{
    \begin{array}{ll}
	  \frac{1}{\sqrt{2\pi}}, & j = 1 \\
	  \frac{\sin(\frac{j}{2}\hat{x}_2)}{\sqrt{\pi}}, & j \text{ even} \\
  	  \frac{\cos(\frac{j-1}{2} \hat{x}_2)}{\sqrt{\pi}}, & j>1 \text{ odd}.
    \end{array}
    \right.
  \end{equation}
  Note that here the weight function is constant, \ie \\$w^\ell_2(\hat{x}_2) \equiv 1$, and the defined trigonometric polynomials are orthonormal in $L^2(\hat{X}_2^\ell)$.
  \item \textit{coordinate $\hat{x}_3,\ldots,\hat{x}_d$}: On the remaining angular dimensions $i=3, \ldots, d$, we employ the usual Gram-Schmidt orthogonalization algorithm on $[0, \pi]$ with weight function $w^\ell_i(\hat{x}_i) = \sin^i(\hat{x}_i)$, based on polynomials.
\end{itemize}
Fortunately, the basis for dimensions $1 < i \leq d$ coincides on every layer $\ell=1,\ldots, L$.
It hence can be computed just once and passed to the individual process handling the current layer. 
Only the basis in the radial dimension needs to be adjusted to $\hat{X}^\ell$.
The parameter $\tau_{\mathrm{GS}}$ collects all tolerance parameters for the applied numerical quadrature and the significant digit length $\tau_{\mathrm{mant}}$.

\paragraph{\textbf{Generation of Samples}}
To generate samples on $\hat{X}^\ell$ with respect to the weight function $w^\ell$, we employ inverse transform sampling.
For this the weight function is rescaled to have unit norm in $L^1(\hat{X}^\ell)$.
Then, the involved inverse cumulative distribution functions can be computed analytically.
We denote the generation process of $N\in\mathbb{N}$ samples as the function
\begin{align}
 \mathcal{S}^\ell&\isdef\left\{\left(\hat{x}^s, \hat{f}_0^{\ell}(\hat{x}^s)\right)\right\}_{s=1}^{N}\nonumber\\ &\qquad \qquad \uparrow \GenerateSamples[\hat{f}^\ell_0, \hat{X}^\ell, w^\ell, N].
\end{align}

\paragraph{\textbf{Reconstruction of a Tensor Train surrogate}}
The VMC reconstruction approach of Section~\ref{sec:low-rank} is summarized in the function 
\begin{equation}
\left\{\hat{F}_{0, i}^{\ell, \mathrm{TT}}\right\}_{i=1}^{d} \leftarrow \ReconstructTT[\mathcal{S}^\ell, \mathcal{P}^\ell, \bs{r}^\ell, \tau_\on{Recon}].
\end{equation}
The tensor components $\hat{F}_{0, i}^{\ell, \mathrm{TT}}$ are associated with the corresponding basis $\mathcal{P}^\ell_i$ to form a rank $\bs{r}^\ell$ extended tensor train as defined in~\eqref{eq:tt representation} and~\eqref{eq:finite tt component}.
The additional parameter $\tau_{\on{Recon}}$ collects all parameters that determine the VMC algorithm.

The method basically involves the optimization of a loss functional over the set of tensor trains with rank (at most) $\bs{r}^\ell$.
In the presented numerical computations we consider a mean-square loss and the respective empirical approximation based on a current sample set $\mathcal{S}^\ell$. 
The tensor optimization, based on a rank adaptive, alternating direction fitting (ADF) algorithm, is implemented in the \texttt{xerus} library~\cite{xerus} and wrapped in the \texttt{ALEA} framework~\cite{alea}.
Additionally, the machine learning framework \texttt{PyTorch}~\cite{paszke2017automatic} can be utilized in \texttt{ALEA} to minimize the empirical cost functional from~\eqref{eq:empirical_cost_functional} by a wide class of state-of-the-art stochastic optimizers.
The latter enables stochastic gradient methods to compute the tensor coefficients as known from machine learning applications.
Having this setting in mind, the actual meaning of the parameter $\tau_{\on{Recon}}$ depends on the chosen optimizer.
In this article we focus on the ADF implementation and initialize \eg the starting rank, the number of iteration of the ADF and a target residual norm.

\begin{algorithm*}
  \begin{algorithmic}
  \REQUIRE{
  \begin{tabular}[t]{llr}
  & Lebesgue target density $f\colon \mathbb{R}^d\to\mathbb{R}_+$ & \eqref{eq:main density} \\
  & tensor spaces $\left\{\hat{X}^\ell\right\}_{\ell=1}^{L}$, with $\hat X^\ell = \bigtimes_{i=1}^d\hat{X}^\ell_i$ & \eqref{eq:local_pert_prior} \\
  & coordinate transformations $\Phi^\ell\colon\hat{X}^\ell \to X^\ell \subset \mathbb{R}^d$ & \eqref{eq:def:rank1} \\
  & $\qquad$ with rank-1 Jacobians $w^\ell\isdef\abs{\det\left[\mathcal{J}_{\Phi^\ell}\right]}\colon \hat{X}^\ell\to \mathbb{R}$ &\\
  & basis dimensions $(\bs{n}^1, \ldots, \bs{n}^L)$, $\bs{n}^\ell\in\mathbb{N}^d$ for $\ell=1,\ldots, L$ & \eqref{eq:finite tt component}\\
  & sample size $N_\ell\in\mathbb{N}$, $\ell=1,\ldots,L$ for level-wise reconstruction & \\
  & tensor train ranks $(\bs{r}^1, \ldots, \bs{r}^L)$, $\bs{r}^\ell\in\mathbb{N}^{d-1}$, for $\ell=1,\ldots, L$ & \eqref{eq:tt representation} \\
  & Gram-Schmidt tolerance parameter $\tau_\on{GS}$ &  \\
  & tensor reconstruction parameter $\tau_\on{Recon}$ & 
  \end{tabular}
  }
  \ENSURE{
  \begin{tabular}[t]{ll}
  & Level-wise low-rank approximation of perturbed prior \\
  \end{tabular}
  }
  \STATE{
  \item[]
  \begin{tabular}{lll}
  Diffeomorphism $\tilde T$&$\leftarrow $&$\ComputeTrafo[f]$\\
  \end{tabular}}
  \item[]
  \FOR{$\ell=1,\ldots, L$, (in parallel)}
  \STATE{\hspace{1ex}$\bullet$ Set transformed \emph{perturbed prior} $\hat{f}^\ell_0(\hat{x}) \isdef \left(f\circ \tilde T\otimes |\operatorname{det}\mathcal{J}_{\tilde{T}}|\right)\circ {\Phi^\ell}(\hat{x})$,\; $\hat{x}\in\hat{X}^\ell$}
  \item[]
  \STATE{
  \begin{tabular}{rll}
  \multicolumn{3}{l}{$\bullet$ Build one-dimensional ONB $\mathcal{P}_i^{\ell}$ of $\mcV_{i, n_i^\ell}\subseteq L^2(\hat{X}^\ell_i, w^\ell_i)$ for $i=1,\ldots, d$} \\
  \multicolumn{3}{l}{$\quad$} \\
  $\mathcal{P}^\ell = \{\mathcal{P}_i^\ell\}_{i=1}^{d}$ &$\leftarrow $&$\GenerateONB[\hat{X}^\ell, \bs{n}^\ell, w^\ell, \tau_\on{GS}]$ \\
  \multicolumn{3}{l}{$\quad$} \\
  \multicolumn{3}{l}{$\bullet$ Generate samples with respect to the weight $w^\ell$} \\
  \multicolumn{3}{c}{$\quad$} \\
  $\mathcal{S}^\ell\isdef\left\{\left(\hat{x}^s, \hat{f}_0^{\ell}(\hat{x}^s)\right)\right\}_{s=1}^{N}$&$\leftarrow$&$ \GenerateSamples[\hat{f}^\ell_0, \hat{X}^\ell, w^\ell, N]$\\
  \multicolumn{3}{l}{$\quad$} \\
  \multicolumn{3}{l}{$\bullet$ Reconstruct TT surrogate $\tilde{f}_{0}^{\ell, \mathrm{TT}}\colon \hat{X}^\ell\to\bbR$} \\
  \multicolumn{3}{l}{$\quad$} \\
  $\left\{\tilde{F}_{0, i}^{\ell, \mathrm{TT}}\right\}_{i=1}^{d}$ & $\leftarrow$ & $\ReconstructTT[\mathcal{S}^\ell, \mathcal{P}^\ell, \bs{r}^\ell, \tau_\on{Recon}]$ \\
  \multicolumn{3}{l}{$\quad$} \\
  \multicolumn{3}{l}{$\bullet$ Equip tensor components with basis} \\
  \multicolumn{3}{l}{$\quad$} \\
  $\hat{f}_{0}^{\ell, \mathrm{TT}}(\hat{x}) $&$ \isdef $&$\sum_{\bs{k}}^{\bs{r}^\ell}\prod_{i=1}^d \hat{f}_{0, i}^{\ell, \mathrm{TT}}[k_{i-1}, k_i](\hat{x}_i)$ \\
  where $\hat{f}_{0, i}^{\ell, \mathrm{TT}}[k_{i-1}, k_i](\hat{x}_i)$ &$ \isdef $ &$ \sum_{j=1}^{n_j^\ell} \hat{F}_{0, i}^{\ell, \mathrm{TT}} [k_{i-1}, \mu_i, k_i] P_{i, j}^\ell(\hat{x}_i)$ 
  \end{tabular}}
  \ENDFOR
  \RETURN{$\left\{\tilde{f_\ell}\right\}_{l=1}^L$}
  \caption{Tensor train surrogate creation of perturbed prior}
  \label{alg:algo_levelrecon}
  \end{algorithmic}
  \end{algorithm*}

  \section{Applications}
  \label{sec:applications}
  In the preceding sections the creation of surrogate models of quite generic probability density functions were developed. 
  Using this, in the following we focus on actual applications where such a representation is beneficial.
  We start with the framework of Bayesian inverse problems with target density~\eqref{eq:main density} corresponding to the Lebesgue posterior density.
  Subsequently, we cover the computation of moments and marginals.
  
  \subsection{Bayesian inversion}
  \label{sec:bayes}
  
  This section is devoted to a brief review of the Bayesian paradigm.
  We recall the general formalism and highlight the notation with the setup of Section~\ref{sec:DensityRepresentation} in mind.
  We closely follow the presentation in~\cite{EMS18} and refer to~\cite{stuart2010inverse,dashti2016bayesian,kaipio2006statistical} for a comprehensive overview.
  
  Let $Y$, $V$ and $\mcY$ denote separable Hilbert spaces equipped with norms $\norm{\cdot}_H$ and inner products $\langle\cdot, \cdot\rangle_H$ for $H\in\{Y, V, \mcY\}$.
  The uncertain quantity $y\in Y$ is tied to the model output $q\in V$ by the \emph{forward map} 
  \begin{equation}
    \label{eq:forwardOp}
    G\colon Y\to V, \quad \theta\mapsto q(y):=G(y).
  \end{equation}
  The usual forward problem reads
  \begin{equation}
    \label{eq:forward problem}
    \text{Given } y\in Y, \text{ find } q\in V.
  \end{equation}
  In contrast to this, the inverse problem is defined by
  \begin{equation}
    \label{eq:inverse problem}
    \text{Given observations of } q, \text{ find } y\in Y.
  \end{equation}
  
  The term \emph{observations} is determined by a bounded linear operator $\mcO\colon V\to \mcY$ that describes the measurement process of the quantity $q$. 
  In practical applications this could be direct observations at sensor points or averaged values from monitoring devices, \eg with $\mcY=\bbR^J$ for some $J\in\bbN$. 
  
  Classically, the (deterministic) quantification problem~\eqref{eq:inverse problem} is not well-posed.
  To overcome this, a problem regularization of some kind is required. 
  The chosen probabilistic approache introduces a random measurable additive noise $\eta\colon (\Omega, \mcU, \bbP) \to (\mcY, \mcB(\mcY))$ with law $\mcN(0, C_0)$ for some symmetric positive definite covariance operator $C_0$ on $\mcY$ to define the noisy measurements
  \begin{equation}
    \label{eq:measurements}
    \delta = (\mcO\circ G) (y) + \eta=: \mcG(y) + \eta\quad\text{where }\mcG\colon Y\to\mcY.
  \end{equation}
  As a consequence, the quantities $y$, $q$ and $\delta$ become random variables over a probability space $(\Omega, \mathcal{F}, \bbP)$ with values in $Y$, $V$ and $\mcY$, respectively.
  In~\cite{stuart2010inverse} mild conditions on the forward operator are derived to show a continuous version of Bayes formula which yields the existence and uniqueness of the Radon-Nikodym derivative of the (posterior) measure $\pi_\delta$ of the conditional random variable $y\vert\delta$ with respect to a prior measure $\pi_0$ of $y$.
  More precisely, by assuming Gaussian $\eta$ and independence with respect to $y$, both measures $\pi_0$ and $\pi_\delta$ on $Y$ are related by the \emph{Bayesian potential} 
  \begin{equation}
    \label{eq:Bayesian potential}
    \Psi(y, \delta) := \frac{1}{2}\langle C_0^{-1}(\delta - \mcG(y)), \delta - \mcG(y)\rangle_{\mcY}
  \end{equation}
  in the sense that 
  \begin{equation}
    \label{eq:Bayes formula} 
    \frac{\mathrm{d}\pi_\delta}{\mathrm{d}\pi_0}(y) = Z^{-1}\exp\left(-\Psi(y, \delta)\right), 
  \end{equation}
  with normalization constant $Z:= \mathbb{E}_{\pi_0} \left[ \exp\left(-\Psi(y, \delta)\right)\right].$
  Note that we interchangeably write $y$ as an element of $Y$ and the corresponding random variable with values in $Y$.
  
  \subsection{Bayesian inversion for parametric PDEs}
  \label{sec:BayesPDE}
  Random partial differential equations (PDEs), \ie \\ PDEs with correlated random data, play an important role in the popular field of Uncertainty Quantification (UQ). 
  As a prominent benchmark example, we consider the ground water flow model, also called the Darcy problem, as \eg examined in~\cite{EGSZ1,EPS17,EMPS18}.
  In this linear second order PDE model, the forward operator $G$ in~\eqref{eq:forwardOp} on some domain $D\subset \bbR^d$, $d=1, 2, 3$ is determined by a forcing term $g\in L^2(D)$ and the random quantity $a(y)\in L^\infty(D)$, which for almost every $y\in Y$ models a conductivity or permeability coefficient.
  The physical system is described by
  \begin{equation}
    \label{eq:Darcy}
    -\ddiv\left(a(y)\nabla q(y)\right) = g \quad\text{in}\; D, \quad q(y)\vert_{\partial D} = 0,
  \end{equation}
  and the solution $q(y)\in V:=H_0^1(D)$ corresponds to the system response $G(y) = q(y)$.
  Pointwise solvability of~\eqref{eq:Darcy} for almost every $y\in Y$ is guaranteed by a Lax-Milgram argument.
  For details we refer to~\cite{schwab2011sparse}. 
  
  For the applications in this article we employ a truncated log-normal coefficient field 
  \begin{equation}
  a(y) = \exp\left(\sum_{k=1}^d a_k y_k \right)
  \end{equation}
  for some fixed $(a_k)_{k=1}^d$ with $a_k\in L^2(D)$ and the image of some random variable with law $\mathcal{N}(0, I)$ denoted by $y=(y_k)_{k=1}^d\in Y$.
  Assume point observations~\eqref{eq:measurements} of $q$ at nodes $\delta = (\delta_1, \ldots, \delta_J)$ in $D$ corresponding to some unknown $q(y^\ast)$, $y^\ast\in Y$.
  We consider the Bayesian posterior density~\eqref{eq:Bayes formula} and set
  \begin{equation}
    \label{eq:pdeposterior density}
    f(y) = Z^{-1}\mathrm{d}\pi_{\delta}(y)\mathrm{d}\pi_0(y)
  \end{equation}
  as the Lebesgue density of the target measure $\pi$ on $Y$ according to~\eqref{eq:main density}.
  
  \subsection{Moment computation}
  \label{sec:moments}
  In this section we discuss the computation of moments for the presented layer-based format with low-rank tensor train approximations. 
  In particular we are interested in an efficient generation of the moment map
  \begin{equation}
       \bm{\alpha} \mapsto \int\limits_{Y} y^{\bm{\alpha}} f(y)\mathrm{d}\lambda(y),\quad \bm{\alpha}=(\alpha_k)_k\in\mathbb{N}_0^d.
  \end{equation} 
  Given some transport $\tilde{T}\colon X\to Y$ with an associated perturbed prior $\tilde{f}_0 = (f\circ \tilde{T})\otimes \abs{\det\mathcal{J}_{\tilde{T}}}$, by an integral transformation it holds
  \begin{equation}
      \int\limits_{Y} y^{\bm{\alpha}}f(y)\mathrm{d}\lambda(y) = 
      \int\limits_{X} \tilde{T}(x)^{\bm{\alpha}} \tilde{f}_0(x)\mathrm{d}\lambda(x).
  \end{equation}

  We fix $1\leq\ell\leq L$ and assume tensor spaces $\hat{X}^\ell, X^\ell$ such that a layer based splitting can be employed to obtain integrals over $X^\ell$ of the form
  \begin{equation}
  \int\limits_{Y} y^{\bm{\alpha}}f(y)\mathrm{d}\lambda(y) = 
      \sum_{\ell=1}^L\int\limits_{X^\ell} \tilde{T}(x)^{\bm{\alpha}} \tilde{f}_0(x)\mathrm{d}x.
  \end{equation}
  Note that we neglect the remaining unbounded layer $X^{L+1}$ since for moderate $\abs{\alpha}$ and $\mathrm{vol}(\bigcup_{\ell=1}^L X^\ell)$ sufficiently large, the contribution to the considered moment does not have a significant influence on the overall approximation.
  Additionally, a rank-1 stable diffeomorphism $\Phi^\ell\colon\hat{X}^\ell\mapsto X^\ell$ is assumed for which there exist univariate functions $\Phi^\ell_{,j}\colon\hat{X}_j^\ell\to X^\ell$ with $\Phi_{,j}^\ell = (\Phi_{i,j}^\ell)_{i=1}^{d}$ and $h_j\colon\hat{X}^\ell_j\to\bbR$ for every $j=1,\ldots, d$, such that
  \begin{equation}
     \label{eq_moment_rank1map}
      \Phi^\ell(\hat{x}) = \prod\limits_{j=1}^d\Phi_{,j}^\ell(\hat{x}_j)\quad \text{and}\quad
      |\det[\mathcal{J}_{\Phi^\ell}](\hat{x})| = \prod\limits_{j=1}^d h_j(\hat{x}_j).
  \end{equation}
  
  \subsubsection{Moments under affine transport}
  \label{sec:moments under affine trafo}
  Let $H=[h_{ki}]_{k,i=1}^d = [h_{1},h_{2},\ldots,h_{d}]\in\mathbb{R}^{d,d}$ be a symmetric positive definite matrix and $M=(M_i)_{i=1}^ d\in\mathbb{R}^d$ such that the considered transport map takes the form
  \begin{equation}
    \tilde{T}(\cdot) = H\cdot + M.
  \end{equation}
  With the multinomial coefficient for $j\in\mathbb{N}$, $\bm{\beta}\in\mathbb{N}_0^d$ with $j = \abs{\bm{\beta}}$ given by
  \begin{equation*}
       \begin{pmatrix} j \\ \bm{\beta}\end{pmatrix} := \frac{j!}{\beta_1!\cdot\ldots\cdot\beta_d!},
  \end{equation*}
  the computation of moments corresponds to the multinomial theorem as seen in the next lemma.
  \begin{lemma}
  \label{lem:multinomial}
  Let $k\in\mathbb{N}$ with $1\leq k\leq d$ and $\alpha_k\in\mathbb{N}_0$. It holds
  \begin{align}
      [H\Phi^\ell(\hat{x})+M)]_k^{\alpha_k} = 
      \sum\limits_{j_k=0}^{\alpha_k}\sum\limits_{|\bm{\beta}_k|=j_k} 
      C_k^H[j_k,\alpha_k,\bm{\beta}_k] \times \nonumber\\
      \qquad \times \prod\limits_{j=1}^{d}\bm{\Phi}_j^{\bm{\beta}_k}(\hat{x}_j),
  \end{align}
  where the high-dimensional coefficient $C_k^H$ is given by
  \begin{equation}
      C_k^H[j_k,\alpha_k,\bm{\beta}_k] := 
      \begin{pmatrix}\alpha_k \\ j_k\end{pmatrix} c_k^{\alpha_k-j_k}
      \begin{pmatrix} j_k \\ \bm{\beta}_k\end{pmatrix} h_k^{\bm{\beta}_k},
  \end{equation}
  with $c_k:=\sum\limits_{i=1}^d h_{ki}M_i$ and
  \begin{equation}
      \label{eq:definition_tensorPhi}
      \bm{\Phi}_j^{\bm{\beta}_k} := [\Phi_{1,j}^\ell(\hat{x}_j),\ldots,\Phi_{d,j}^\ell(\hat{x}_j)]^{\bm{\beta}_k}.
  \end{equation}
  \end{lemma}
  \begin{proof}
  Note that
  \begin{align*}
      [H\Phi^\ell(\hat{x})+M)]_k^{\alpha_k} &= 
     \sum\limits_{j_k=0}^{\alpha_k}\begin{pmatrix}\alpha_k \\ j_k\end{pmatrix} c_k^{\alpha_k-j_k}
  \times \\
  &\quad \times
  \left(\sum\limits_{i=1}^d h_{ki}\prod\limits_{j=1}^{d}\Phi_{ij}^\ell(\hat{x}_j)\right)^{j_k}.
  \end{align*}
  The statement follows by the multinomial theorem since 
  \begin{align*}
      \left(\sum\limits_{i=1}^d h_{ki}\prod\limits_{j=1}^{d}\Phi_{ij}^\ell(\hat{x}_j)\right)^{j_k}
      &= 
  \sum\limits_{|\bm{\beta}_k| = j_k}\begin{pmatrix} j_k \\
  \bm{\beta}_k\end{pmatrix}
  \left(\prod\limits_{i=1}^{d} h_{ki}^{(\bm{\beta}_k)_i}\right) \times \\
  &\qquad \times
  \left(\prod\limits_{j=1}^d \prod\limits_{i=1}^d\Phi_{ij}^\ell(\hat{x}_j)^{(\bm{\beta}_k)_i}\right).
  \end{align*}
  \end{proof}
  Generalizing Lemma \ref{lem:multinomial} to multiindices $\bm{\alpha}\in\mathbb{N}_0^d$ yields
  \begin{align}
      [H\Phi^\ell(\hat{x})+M)]^{\bm{\alpha}} &=
      \sum\limits_{\bm{j}=0}^{\bm{\alpha}}
     \sum\limits_{(|\bm{\beta}_k|)_k
      =\bm{j}}
      \left(\prod\limits_{k=1}^d C_k^H[j_k,\alpha_k,\bm{\beta}_k]\right) \times \nonumber\\
      &\quad\times
      \prod\limits_{j=1}^{d}\bm{\Phi}_j^{\sum\limits_{k=1}^d\bm{\beta}_k}(\hat{x}_j),
  \end{align}
  where  $\sum\limits_{(|\bm{\beta}_k|)_k
      =\bm{j}}:=\sum\limits_{|\bm{\beta}_1|=j_1}\ldots \sum\limits_{|\bm{\beta}_d|=j_d}$ is used.
  
  Exploiting the layerwise TT representation of $\hat{f}_\ell$ from~\eqref{eq:full_discrete_representation} and using the rank-1 stable map \eqref{eq_moment_rank1map}, the high-dimensional integral over $X^\ell$ reduces to 
  \begin{align}
     \label{eq:separateIntegrals}
      &\phantom{=}\int\limits_{X_\ell}\tilde{T}(x)^{\bm{\alpha} }\tilde{f}_0(x)\mathrm{d}\lambda(x) \nonumber\\
      &= 
       \sum\limits_{\bm{j}=0}^{\bm{\alpha}}
     \sum\limits_{(|\bm{\beta}_k|)_k
      =\bm{j}}\left(\prod\limits_{k=1}^d C_k^H[j_k,\alpha_k,\bm{\beta}_k]\right)\nonumber\\
     &\qquad \times \sum\limits_{\bm{k}=\bm{0}}^{\bm{r}_\ell }
     \prod\limits_{i=1}^d
     \int\limits_{\hat{X}_i} \left[\hat{f}_{\ell,i}[k_{i-1},k_{i}] \otimes
  \bm{\Phi}_i^{\sum\limits_{k=1}^d\bm{\beta}_k}\otimes h_i\right]\!(\hat{x}_i)\,\mathrm{d}{\hat{x}}_i.
  \end{align}
  Note that the right-hand side is composed via decoupled one dimensional integrals only.
  We point out that while the structure is simplified, the definition of $\bm{\Phi}_j$ in~\eqref{eq:definition_tensorPhi} a priori results in several integrals (indexed by $\sum\limits_{k=1}^d\bm{\beta}_k$).
  These integrals, whose number depends on the cardinality of $\bm{\alpha}$, have to be computed.
  This simplifies further in several cases, \eg when $\Phi^\ell$ transforms the spherical coordinate system to Cartesian coordinates.
  
  \paragraph{\textbf{Moment computation using spherical coordinates}}
  In the special case that $\Phi^\ell$ is the multivariate polar transformation of Example~\ref{ex:polar_coords}, the number of distinct computation of integrals from \eqref{eq:separateIntegrals} reduces significantly.
  Recall that $\hat{x}_1 = \rho$, $\hat{x}_{2:d} = \bm{\theta}=(\theta_0,\ldots,\theta_{d-2})$ and let $\beta_i^k:=(\bm{\beta}_k)_i$ be the $i$-th entry of $\bm{\beta}_k$.
  We find that
  \begin{align}
  \label{eq:polar_exponents_1}
      \bm{\Phi}_1^{\sum\limits_{k=1}^d\bm{\beta}_k}(\rho) &= \rho^{|\bm{j}|}, \\
      \bm{\Phi}_2^{\sum\limits_{k=1}^d\bm{\beta}_k}(\theta_0)
      &=\cos^{\left(\sum\limits_{k=1}^d\beta_1^k\right)}(\theta_0) 
  \sin^{\left(\sum\limits_{k=1}^d\beta_2^k\right)}(\theta_0),\\
  \label{eq:polar_exponents_3}
      \bm{\Phi}_{i+2}^{\sum\limits_{k=1}^d\bm{\beta}_k}(\theta_{i})
      &= \sin^{\left(\sum\limits_{l=1}^i\sum\limits_{k=1}^d \beta_l^k\right)}(\theta_i) 
  \cos^{\left(
  \sum\limits_{k=1}^d\beta_{i+1}^k
  \right)}(\theta_i).
  \end{align}
  for $1 \leq i \leq d-2 .$
  
  The exponential complexity due to the indexing by $\sum_{k=1}^d\bm{\beta}_k$ reduces to linear complexity in $|\bm{\alpha}|$.
  More precisely, the amount of exponents in \eqref{eq:polar_exponents_1} 
  - \eqref{eq:polar_exponents_3} is linear in the dimensions since the sums only depend on $|\bm{\alpha}|$, leading to $\mathcal{O}(|\bm{\alpha}|d)$ different integrals that may be precomputed for each tuple $(k_{i-1},k_{i})$.
  This exponential complexity in the rank vanishes in the presence of an approximation basis associated with each coordinate dimension as defined in Section~\ref{sec:low-rank}.

  \subsection{Computation of marginals}
  \label{sec:marginals}
  In probability theory and statistics, marginal distributions and especially marginal probability density functions provide insights into an underlying joint density by means of lower dimensional functions that can be visualized.
  The computation of marginal densities is a frequent problem encountered \eg in parameter estimation and when using sampling techniques since histograms and corner plots provide easy access to (in general high-dimensional) integral quantities.
  
  In contrast to the Markov chain Monte Carlo algorithm, the previously presented method of a layer based surrogate for the Lebesgue density function $f\colon Y=\mathbb{R}^d \to \bbR$ allows for a functional representation and approximation of marginal densities without additional evaluations of $f$.
  
  For simplicity, for $y\in Y$ and $i=1, \ldots, d$ define $y_{-i} = (y_1, \ldots, y_{i-1}, y_{i+1}, \ldots y_d)$ as the marginalized variable where the $i$-th component is left out and $f(y_{-i}, y_i) \isdef f(y)$.
  Then, for given $i=1,\ldots, d$, the $i$-th marginal density reads
  \begin{equation}
    \label{eq:marginalpdf}
    \mathrm{d}f_i(y_i) \isdef \int_{\bbR^{d-1}} f(y_{-i}, y_i) \mathrm{d}\lambda(y_{-i}).
  \end{equation}
  Computing this high-dimensional integral by quadrature or sampling is usually infeasible and the transport map approach as given by~\eqref{eq:transport_moment_equation} fails since the map $T\colon X\to Y$ cannot be used directly in~\eqref{eq:marginalpdf}.
  Alternatively, we can represent $\mathrm{d}f_i\colon \bbR \to \bbR$ in a given orthonormal basis $\{\varphi_j\}_{j=1}^{N_\varphi}$ and consider
  \begin{equation}
    \mathrm{d}f_i(y_i) = \sum_{j=1}^{N_{\varphi}} \beta_j \varphi_j(y_i),
  \end{equation}
  where $\beta_j$, $j=1, \ldots, N_\varphi$ denotes the $L^2(\bbR)$ projection coefficient
  \begin{equation}
    \label{eq:marginalL2Project coef}
    \beta_j \isdef \int_{\bbR} \varphi_j(y_i) \mathrm{d}f_i(y_i) \mathrm{d}\lambda(y_i).
  \end{equation}
  With this the marginalisation can be carried out similar to the computations in Section~\ref{sec:moments}.
  
  A convenient basis is given by monomials since~\eqref{eq:marginalL2Project coef} then simplifies to 
  \begin{equation}
    \label{eq:marginalL2ProjectMonomial}
    \beta_j = \int_{\bbR^d} y_i^j f(y)\mathrm{d}\lambda(y).
  \end{equation}
  This is the moment corresponding to the multiindex $\alpha=(\alpha_k)_{k=1}^d\in\bbN^d$ with $\alpha_k = \delta_{k, j}$.
  Alternatively, indicator functions may be considered in the spirit of histograms.
  
  \subsection{More general quantities of interest}
  \label{sec:qoi}
  One is frequently concerned with efficiently computing the expectation of some quantity of interest (QoI) $Q\colon Y\to\bbR$
  \begin{equation}
    \label{eq:int_qoi}
    \mathbb{E}\left[Q\right] = \int_{Y} Q(y) f(y) \mathrm{d}\lambda(y).
  \end{equation}
  We discussed this issue for moments in Section~\ref{sec:moments} and basis representations of marginals in Section~\ref{sec:marginals}.
  In those cases the structure of $Q$ allows for direct computations of the integrals via tensor contractions.
  For more involved choices of the QoI we suggest a universal sampling approach by repeated evaluation of the low-rank surrogate.
  More precisely, by application of the integral transformation we can approximate
  \begin{equation}
    \label{eq:trafoqoi}
    \mathbb{E}\left[Q\right] \approx \sum_{\ell=1}^L\int_{\hat{X}^\ell} Q\circ \tilde{T}\circ \Phi^\ell(\hat{x}) \tilde{f}_{0}^{\ell, \mathrm{TT}}(\hat{x}) \abs{\mathrm{det}\left[\mathcal{J}_{\Phi^\ell}\right](\hat{x})}\mathrm{d}\lambda(\hat{x})
  \end{equation}
  and replace the integrals over $\hat{X}^\ell$ by Monte Carlo estimates with samples according to the (normalized) weight $\abs{\mathrm{det}\left[\mathcal{J}_{\Phi^\ell}\right]}$.
  Those samples can be obtained by uniform sampling on the tensor spaces $\hat{X}^\ell$ and the inverse transform approach as mentioned in the paragraph \textbf{Generating Samples} of Section~\ref{sec:Algorithm}.
  Alternatively, efficient MCMC sampling by marginalization can be employed~\cite{weare2007efficient}.

  \section{Numerical validation and applications}
  \label{sec:numerical experiments}
  This section is devoted to a numerical validation of the proposed Algorithm~\ref{alg:algo_levelrecon} using various types of transformations $T$ while employing it with practical applications.
  We focus on three example settings.
  The first consists of an artificial Gaussian posterior density, which could be translated to a linear forward model and Gaussian prior assumptions in the Bayesian setting.
  Second, we study the approximation under non-exact transport and conclude as a third setting with an actual Bayesian inversion application governed by the log-normal Darcy flow problem of Section~\ref{sec:BayesPDE}.
  
  \subsection{Validation experiment 1: Gaussian density}
  In this experiment we confirm the theoretical results from Section~\ref{sec:error estimates} and verify the numerical algorithm. 
  Even though the examined approximation of a Gaussian density is not a challenging task for the proposed algorithm, it can be seen as the most basic illustration revealing the possible rank-1 structure of the perturbed prior under optimal transport.
  
  We consider the posterior density determined by a Gaussian density with covariance matrix $\Sigma\in\mathbb R^{d, d}$ and mean $\mu\in\mathbb R^d$ as 
  \begin{equation}
      \label{eq:ne_gaussian density} 
      \frac{\mathrm{d}\pi}{\mathrm{d} \lambda}(x) = f(x) = C \exp\left(-\frac{1}{2}\norm{x - \mu}^2_{\Sigma^{-1}}\right),
  \end{equation}
  where $C=(2\pi)^{-\nicefrac{d}{2}}\det\Sigma^{-\nicefrac{1}{2}}$ is the normalizing factor of the multivariate Gaussian. 
  We set the covariance operator such that the Gaussian density belongs to uncorrelated random variables, i.e. $\Sigma$ exhibits a diagonal structure, and it holds for some $0 < \sigma \ll 1$ that $\Sigma = \sigma^2 I$.
  This Gaussian setting has several benefits as a validation setting. 
  On the one hand, we have explicit access to the quantities that are usually of interest in Bayesian inference like the mean, covariance, normalization constant and marginals. 
  On the other hand, the optimal transport to a standard normal density
  \begin{equation}
    \label{eq:std normal} 
    f_0(x) = (2\pi)^{-\nicefrac{d}{2}} \exp\left(-\frac{1}{2}\norm{x}^2\right)
  \end{equation} 
  is given by an affine linear function, defined via mean $\mu$ and covariance $\Sigma$ as proposed in Remark~\ref{rem:affine trafo choice}.
  We subsequently employ the multivariate polar transformation from Example~\ref{ex:polar_coords} and expect a rank-1 structure in the reconstruction of the local approximations of the (perturbed) prior.
  
  The remainder of this section considers different\\ choices of $\sigma\in\mathbb R$ and $d\in\mathbb N$ and highlights the stability of our method under decreasing variance (\ie with higher density concentration) and increasing dimension.
  The approximations are compared with their exact counterparts.
  More specifically, the error of the normalization constant is observed, namely
  \begin{equation}
    \label{eq:err_Z}
    \err_Z \isdef \abs{1 - Z_h},
  \end{equation}
  the relative error of the mean and covariance in the Euclidean and Frobenius norms $|\cdot|_2$ and $|\cdot|_{\mathrm{F}}$,
  \begin{equation}
    \label{eq:err_mu_sigma}
    \err_\mu \isdef |\mu - \mu_h|_2 |\mu|_{2}^{-1}, \quad 
    \err_\Sigma \isdef |\Sigma - \Sigma_h|_{\mathrm{F}}|\Sigma|_{\mathrm{F}}^{-1},
  \end{equation}
  and the deviation in terms of the Kullback-Leibler divergence~\eqref{eq:Def KL}.
  Computing the Kullback-Leibler divergence is accomplished by Monte Carlo samples $(x_i)_{i=1}^{N_\on{KL}}$ of the posterior (\ie in this case the multivariate Gaussian posterior) to compute the empirical approximation 
  \begin{align}
    \label{eq:empirical_KL}
    \mathrm{d}_\on{KL}(\pi, \pi_h) &= \int_{\mathbb{R}^d} \log\left(\frac{f(x)}{f_h(x)}\right) f(x)\mathrm{d}\lambda(x)\nonumber\\ &\approx \frac{1}{N_\on{KL}} \sum_{i=1}^{N_\on{KL}} \log\left(\frac{f(x_i)}{f_h(x_i)}\right).
  \end{align}
  The index $h$ generically denotes the employed approximation~\eqref{eq:full_discrete_representation}.
  In the numerical experiments the convergence of these error measures is depicted with respect to the number of calls to the forward model (\ie the Gaussian posterior density), the discretization of the radial component $\rho\in [0, \infty)$ in the polar coordinate system and the number of samples on each layer $X^\ell$, $\ell=1, \ldots, L$, for fixed $L\in\mathbb{N}$.
  
  \begin{table*}
    \pgfplotstableread[col sep=comma]{Z_err_dvar.dat}\zerr
    \begin{center}
     \pgfplotstabletypeset[
     col sep=comma,
     every head row/.style={before row=\toprule,after row=\midrule},
     every last row/.style={after row=\bottomrule},
     columns/d/.style={
       column name={dimension},
     },
     columns/1e-2/.style={
       column name={$\sigma^2=10^{-2}$},
     },
     columns/1e-4/.style={
       column name={$\sigma^2=10^{-4}$},
     },
     columns/1e-6/.style={
       column name={$\sigma^2=10^{-6}$},
     },
     columns/1e-8/.style={
       column name={$\sigma^2=10^{-8}$},
     },
     ]{\zerr}
     \caption{Numerical approximation of $Z$ in the Gaussian example. 
     Error of the normalization constant computed via a TT surrogate to $Z=1$.}
     \end{center}
     \label{table:gauss_z}
  \end{table*}  
  
  In Table~\ref{table:gauss_z} $\err_Z$ is depicted for different choices of $\sigma$ and $d$. 
  The experiment comprises radial discretizations $0 = \rho_0 < \rho_1 < \ldots < \rho_{L} = 10$ with $L=19$ equidistanly chosen layers and $1000$ samples of $f_0$ on each resulting subdomain $X^\ell$. 
  The generated basis~\eqref{eq:GenerateOnb} contains polynomials of maximal degree 7 in $\rho_\ell$, $\ell=0, \ldots, L$, and constant functions in every angular direction.
  The choice of constant functions relies on the assumption that the perturbed prior that has to be approximated corresponds to the polar transformation of~\eqref{eq:std normal}, which is a function in $\rho$ only. 
  Additional numerical test show that even much fewer samples and a larger basis lead to the assumed rank-1 structure.
  It can be observed that the approximation quality of $Z$ is invariant under the choice of $\sigma$ and fairly robust with the dimension $d$, which is expected since the transformation is exact and the function to reconstruct is a rank-1 object.

  \begin{figure*}
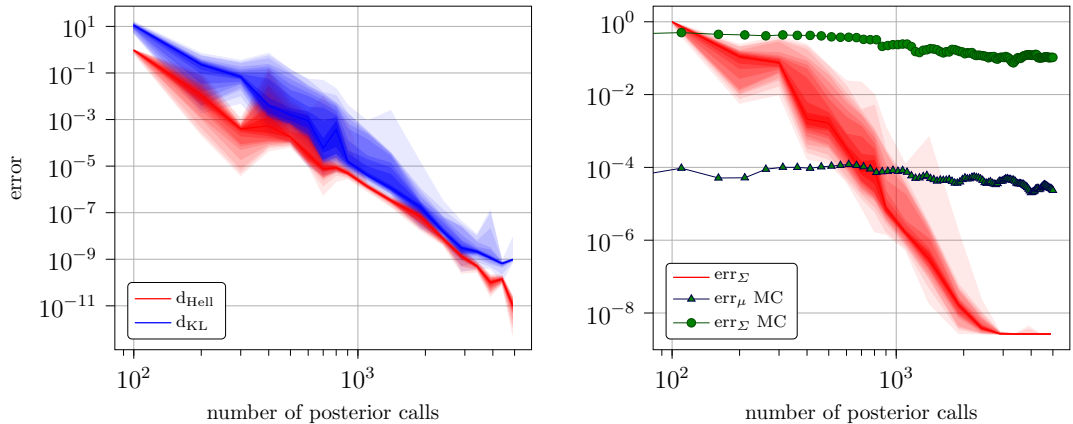

    \begin{center}

    \end{center}

    \caption{Gaussian density example with $d=10$, mean $\mu=\bm{1}$ and noise level $\sigma=10^{-7}$. 
  Tensor reconstructions are repeated with $50$ random sample sets to show quantile range from $5\%-95\%$ (pastel) to the median (bold).
  Hellinger distance and Kullback-Leibler divergence are shown (left) and the relative covariance error together with MCMC results for mean and covariance are given (right).
    }
    \label{fig:gauss_calls}
  \end{figure*}
  
  In Figure~\ref{fig:gauss_calls} we compare the number of calls of the posterior density $f$ explicitly.
  Here, the presented low-rank surrogate is again constructed on an increasing number of layers, whereas the Monte Carlo estimates are computed using a Markov Chain Monte Carlo algorithm and subsequent empirical integration of the error quantity.
  By taking $100$ samples for each added layer, we observe fast convergence in comparison to the slow MC approach\footnote{We emphasize that we just use a baseline MCMC algorithm for comparison. Although more sophisticated MCMC methods could show a more favorable convergence behavior, the fundamental qualitative difference due to entirely different approximation approaches would still persist.}. 
  To further analyse the reconstruction stability we repeat the experiment $50$ times and show empirical quantiles.
  The light area represents the $90\%$ quantile of the distribution and the bold line is the median.
  We observe a larger variance for the Kullback-Leibler divergence in contrast to the Hellinger distance.
  
  Note that we do not show the tensor approximation result for $\mathrm{err}_\mu$ since already for the first case of only $100$ evaluations of the posterior (which corresponds to a single layer) we obtain results close to machine precision.
  This is due to the choice of an exact transport, already containing the correct mean, and how the mean is computed in the presented format, see Section~\ref{sec:moments under affine trafo}.
  In short, the approximation cancels due to the normalization and only the correct mean of the transport formula is left. 
  Concerning the stagnation of $\mathrm{err}_\Sigma$ we suspect a precision problem in the computation, which is confirmed by the small variance.
  Nevertheless, an approximation of around seven magnitudes smaller than MCMC  for the covariance is achieved. 
  
  \subsection{Validation experiment 2: Perturbation of exact transport}
  In the following experiment we consider a so-called ``banana example'' as posterior density, see e.g.~\cite{marzouk2016introduction}.
  Let $f_0$ be the density of a standard normal Gaussian measure and let $T_{\Sigma}$ be the affine transport of $\mathcal{N}(0,I)$ to the Gaussian measure $\mathcal{N}(0,\Sigma)$.
  Furthermore, set
  \begin{equation}
  T_2(x) = \begin{pmatrix} x_1 \\ x_2 - (x_1^2 + 1) \end{pmatrix}.
  \end{equation}
  The exact transport $T$ from $\mathcal{N}(0,I)$ to the curved and concentrated banana distribution with density $f$ is then given by
  \begin{equation}
  T(x) = T_2 \circ T_{\Sigma}(x), \quad \Sigma = \begin{pmatrix} 1 & 0.9 \\ 0.9 & 1 \end{pmatrix}.
  \end{equation}
  Note that the employed density can be transformed into a Gaussian using a quadratic transport function.
  For this experiment, we employ transport maps $\tilde{T}$ of varying accuracy for the pull-back of the posterior density to a standard Gaussian.
  In particular we use an approximation $\tilde{T}_1$ (obtained with~\cite{TM}) of the optimal affine transport $T_1$, and the quadratic transport $T$ to build an approximation $\tilde{T}$ given as convex combination
  \begin{equation}
      \label{eq:numerik_convexpertT}
      \tilde{T}(x) =  (1-t)\;\tilde{T}_1(x) + t \; T(x),\quad t\in[0,1].
  \end{equation}
  For $t=1$, the transport map is optimal since it generates the desired reference density.
  For $0 \leq t < 1$ a perturbed prior density is obtained with strength of perturbation determined by $t$.
  The impact of the perturbed transport is visualized in Figure~\ref{fig:convex_illustration}.
  \begin{figure*}
      \centering
      \includegraphics[width = 0.75\linewidth]{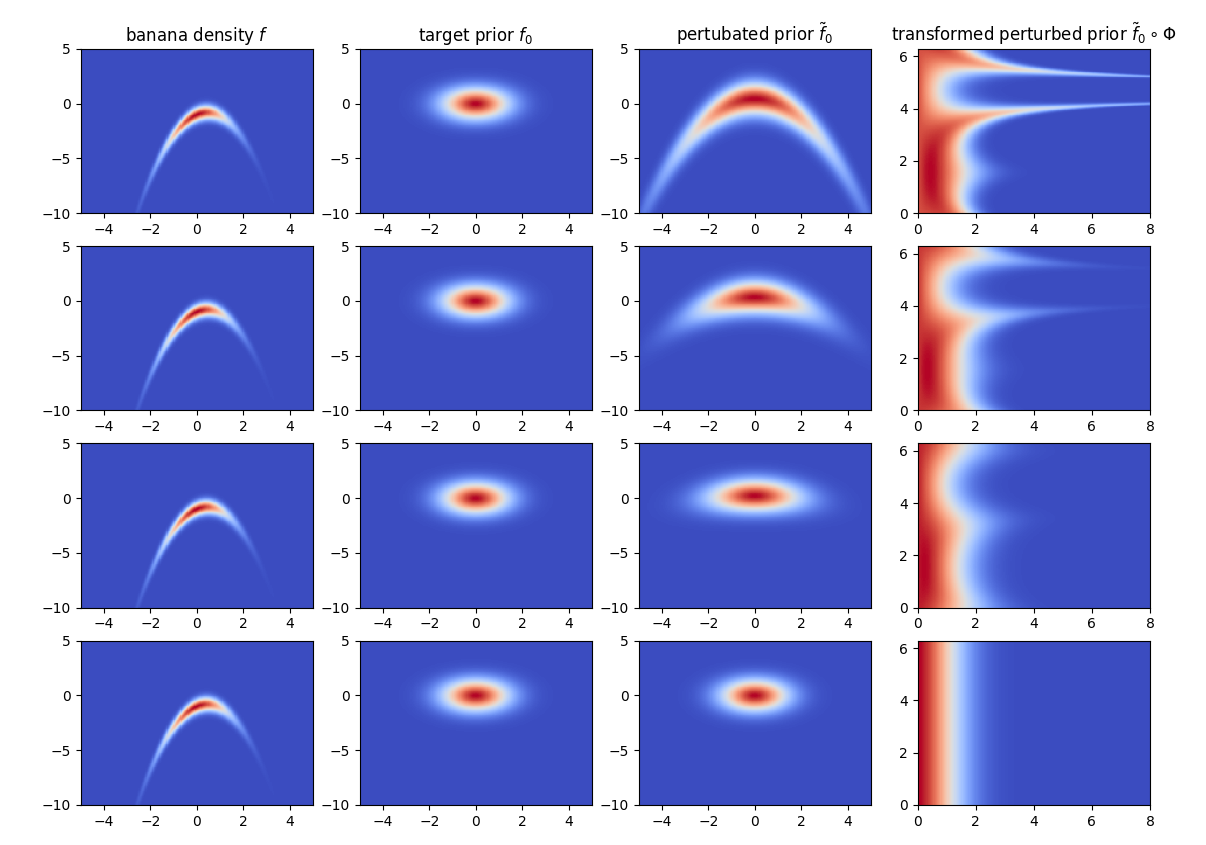}
      \caption{Illustration of the effect of different transports in~\eqref{eq:numerik_convexpertT} for $t=0,0.25,0.5,1.$ (top to bottom).}
      \label{fig:convex_illustration}
  \end{figure*}
  It can be observed that the transformed perturbed prior is not of rank-1 as long as the transformation is inexact.
  Furthermore, the difference between the target prior and the perturbed prior is eminent, which implies that \eg a Laplace approximation to the considered banana density would neglect possible important features of the distribution.

  In Figure~\ref{fig:convex_example} we illustrate the impact of an inexact transport on the approximation results in terms of $\err_\mu$ and $\err_\Sigma$.
  For the considered target density, mean and covariance are known analytically and hence no reference sampling has to be carried out.
  We additionally employ an MCMC sampling to show the improvement due to the additional low-rank reconstruction.
  For the optimal transport map one observes that the surrogate reconstruction reduces to the approximation of a rank-1 Gaussian density, which can be done efficiently with few evaluations of $f$.
  If the transport is only linear and inaccurate, results comparable to MCMC are achieved.
  For a more accurate transport, the low-rank reconstruction leads to drastically improved estimates.
    
  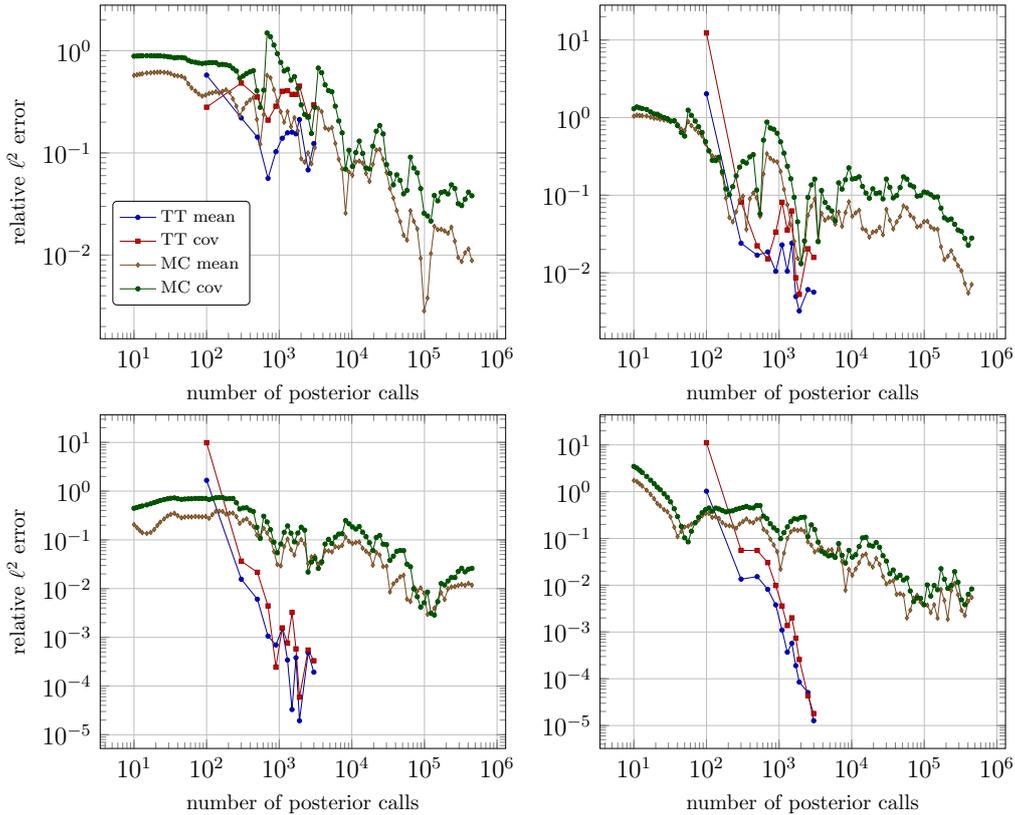
\begin{figure*}
   \begin{center}
    \begin{tikzpicture}[tikzpic options]
    \begin{loglogaxis}[zmeancoverrCallsQuada plot]
  
      \pgfplotstableread[col sep=comma]{numerics/convex/alpha_0_0/tt.txt}\tta
  
      \addplot table[x=calls,y expr=\thisrow{mean_err_l2_rel}] {\tta}; 		
      \addlegendentry{TT mean};
      \addplot table[x=calls,y expr=\thisrow{cov_err_l2_rel}] {\tta}; 		
      \addlegendentry{TT cov};
      
      \pgfplotstableread[col sep=comma]{numerics/convex/alpha_0_0/mc.txt}\mc
  
      \addplot table[x=sample,y expr=\thisrow{mean_err_l2_rel}] {\mc}; 		
      \addlegendentry{MC mean};
      \addplot table[x=sample,y expr=\thisrow{cov_err_l2_rel}] {\mc}; 		
      \addlegendentry{MC cov};
      
    \end{loglogaxis}
    \end{tikzpicture}
    \begin{tikzpicture}[tikzpic options]
    \begin{loglogaxis}[zmeancoverrCallsQuadb plot]
      
      \pgfplotstableread[col sep=comma]{numerics/convex/alpha_0_25/tt.txt}\tta
  
      \addplot table[x=calls,y expr=\thisrow{mean_err_l2_rel}] {\tta}; 		
      \addlegendentry{TT mean $\alpha=0.25$ (100)};
      \addplot table[x=calls,y expr=\thisrow{cov_err_l2_rel}] {\tta}; 		
      \addlegendentry{TT cov $\alpha=0.25$ (100)};
  
      \pgfplotstableread[col sep=comma]{numerics/convex/alpha_0_25/mc.txt}\mc
  
      \addplot table[x=sample,y expr=\thisrow{mean_err_l2_rel}] {\mc}; 		
      \addlegendentry{MC mean $\alpha=0.25$)};
      \addplot table[x=sample,y expr=\thisrow{cov_err_l2_rel}] {\mc}; 		
      \addlegendentry{MC cov $\alpha=0.25$};
      \legend{}				
    \end{loglogaxis}
    \end{tikzpicture}
    
    \begin{tikzpicture}[tikzpic options]
    \begin{loglogaxis}[zmeancoverrCallsQuada plot]
  
      \pgfplotstableread[col sep=comma]{numerics/convex/alpha_0_5/tt.txt}\tta
  
      \addplot table[x=calls,y expr=\thisrow{mean_err_l2_rel}] {\tta}; 		
      \addlegendentry{TT mean $\alpha=0.5$ (100)};
      \addplot table[x=calls,y expr=\thisrow{cov_err_l2_rel}] {\tta}; 		
      \addlegendentry{TT cov $\alpha=0.5$ (100)};
      
      \pgfplotstableread[col sep=comma]{numerics/convex/alpha_0_5/mc.txt}\mc
  
      \addplot table[x=sample,y expr=\thisrow{mean_err_l2_rel}] {\mc}; 		
      \addlegendentry{MC mean $\alpha=0.5$)};
      \addplot table[x=sample,y expr=\thisrow{cov_err_l2_rel}] {\mc}; 		
      \addlegendentry{MC cov $\alpha=0.5$};
      \legend{}
    \end{loglogaxis}
    \end{tikzpicture}
    \begin{tikzpicture}[tikzpic options]
    \begin{loglogaxis}[zmeancoverrCallsQuadb plot]
      
      \pgfplotstableread[col sep=comma]{numerics/convex/alpha_1_0/tt.txt}\tta
  
      \addplot table[x=calls,y expr=\thisrow{mean_err_l2_rel}] {\tta}; 		
      \addlegendentry{TT mean $\alpha=1$ (100)};
      \addplot table[x=calls,y expr=\thisrow{cov_err_l2_rel}] {\tta}; 		
      \addlegendentry{TT cov $\alpha=1$ (100)};
      
      \pgfplotstableread[col sep=comma]{numerics/convex/alpha_1_0/mc.txt}\mc
  
      \addplot table[x=sample,y expr=\thisrow{mean_err_l2_rel}] {\mc}; 		
      \addlegendentry{MC mean $\alpha=1$)};
      \addplot table[x=sample,y expr=\thisrow{cov_err_l2_rel}] {\mc}; 		
      \addlegendentry{MC cov $\alpha=1$};
      \legend{}
    \end{loglogaxis}
    \end{tikzpicture}
  
    \caption{Convex combination of affine and quadratic transport for the banana posterior.
    Affine linear map ($t=0$ top left), transport with $t=0.25$ (top right), $t=0.5$ (bottom left) and exact quadratic map ($t=1$, bottom right).
    Error quantities $\err_\mu$ and $\err_\Sigma$ for the employed tensor train surrogate and a MCMC approximation in terms of the number of calls to the posterior function.
    The surrogate is reconstructed from $100$ samples per layer yielding a tensor with radial basis up to polynomial degree $9$ and Fourier modes up to degree $20$. 
    }
    \label{fig:convex_example}
  \end{center}
  \end{figure*}

  \subsection{Bayesian inversion with log-normal Darcy forward model}
  \label{sec:Experiment Bayes with PDE}
  Revisiting the example of Section~\ref{sec:BayesPDE}, we consider the elliptic diffusion problem with a log-normal random parametric permeability coefficient.
  The considered field in $L^2(Y, L^\infty(D))$ takes the form 
  \begin{equation}
    a(x, y) = \exp\left(\sum_{i=1}^d a_i(x) y_i \right),
  \end{equation}
  where the $y_i$ correspond to random variables with law $\mathcal{N}(0, 1)$ and $L^2(D)$ orthonormal functions $a_i$ being planar Fourier cosine modes.
  A detailed description and an adaptive Galerkin approach to solve the forward problem can be found in~\cite{EMPS18}.
  For the inverse problem, the observation operator is modelled by $J=144$ equidistantly distributed observations in $D=[0, 1]^2$ of the solution $q(y^\ast)\in H_0^1(D)$ for some $y^\ast\in Y=\bbR^d$, which is drawn from a standard normal distribution.
  Additionally, the observations are perturbed by a centered Gaussian noise with covariance $\sigma I$ with $\sigma = 10^{-7}$.
  
  To obtain the desired relative error quantities, we employ reference computations that involve adaptive quadrature for the two dimensional example in Figure~\ref{fig:poisson mean} and Markov Chain Monte Carlo integration with $10^6$ steps of the chain and a burn-in time of $1000$ samples for the experiment in Figure~\ref{fig:poisson d}. 
  For the reconstruction algorithm an affine linear transport is estimated by Hessian information of the log-likelihood and on every layer we employ $100$ samples.
  The respective relative errors are displayed in Figure~\ref{fig:poisson d}.
  
  The stagnation of the graphs in Figure~\ref{fig:poisson mean} is on the one hand governed by the observation noise and on the other hand explicable by a non-optimal reference solutions since the TT approximation yields results equivalent to an adaptive quadrature when taking $L=5$ layers of refinement and thus a total of $500$ samples.
  
  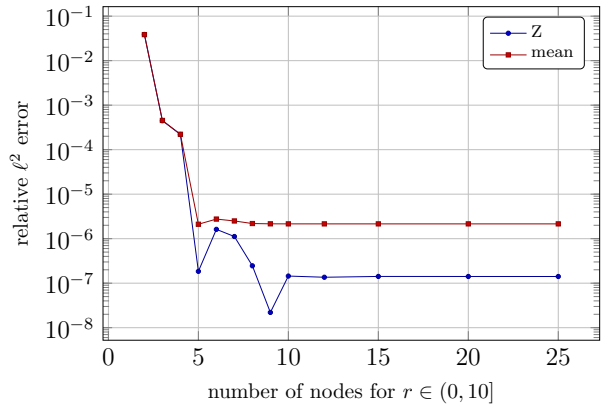
\begin{figure*}
  \begin{center}
    \begin{tikzpicture}[tikzpic options]
    \begin{semilogyaxis}[zmeancoverr plot]
  
      \pgfplotstableread[col sep=comma]{radi_err_poisson_d2.dat}\radierr
  
      \addplot table[x=r,y expr=\thisrow{z}] {\radierr}; 		
      \addlegendentry{Z};
      \addplot table[x=r,y expr=\thisrow{mean}] {\radierr}; 		
      \addlegendentry{mean};
      
    \end{semilogyaxis}
    \end{tikzpicture}
  \end{center}
    \caption{Comparison of the computed reference and the low-rank surrogate of (1) normalization constant ($\err_Z$), and (2) mean ($\err_\mu$). 
    For the Darcy setting with $d=2$ we observe $144$ nodes in the physical domain. 
    The measurements are perturbed by Gaussian noise with deviation $\eta=1e-7$.
    We employ an adaptive quadrature in the two dimensional space to obtain the reference quantities.
    The stagnation of the graphs are due to non-optimal reference solutions. More precisely, the TT approximation yields equivalent results to adaptive quadrature when taking 5 nodes of refinement.}
    \label{fig:poisson mean}
  \end{figure*}
  
  \begin{figure*}
  \begin{center}
    \begin{tikzpicture}[tikzpic options]
    \begin{loglogaxis}[zmeancoverrCallsLeft plot]
  
      \pgfplotstableread[col sep=comma]{numerics/poisson/d2/tt.txt}\tt
      \pgfplotstableread[col sep=comma]{numerics/poisson/d2/mc.txt}\mc
  
      \addplot table[x=calls,y expr=\thisrow{mean_err_l2_rel}] {\tt}; 		
      \addlegendentry{rel mean err TT};
      \addplot table[x=sample,y expr=\thisrow{mean_err_l2_rel}] {\mc}; 		
      \addlegendentry{rel mean err MC};
      \addplot table[x=calls,y expr=\thisrow{cov_err_l2_rel}] {\tt}; 		
      \addlegendentry{rel cov err TT};
      
      \addplot table[x=calls,y expr=\thisrow{kl}] {\tt}; 		
      \addlegendentry{KL distance};
    \end{loglogaxis}
    \end{tikzpicture}
    \begin{tikzpicture}[tikzpic options]
    \begin{loglogaxis}[zmeancoverrCallsRight plot]
  
      \pgfplotstableread[col sep=comma]{numerics/poisson/d10/tt.txt}\tt
      \pgfplotstableread[col sep=comma]{numerics/poisson/d10/mc.txt}\mc
  
      \addplot table[x=calls,y expr=\thisrow{mean_err_l2_rel}] {\tt}; 		
      \addlegendentry{rel mean err TT};
      \addplot table[x=sample,y expr=\thisrow{mean_err_l2_rel}] {\mc}; 		
      \addlegendentry{rel mean err MC};
      \addplot table[x=calls,y expr=\thisrow{cov_err_l2_rel}] {\tt}; 		
      \addlegendentry{rel cov err TT};
      
      \addplot table[x=calls,y expr=\thisrow{kl}] {\tt}; 		
      \addlegendentry{KL distance};
      \legend{}
    \end{loglogaxis}
    \end{tikzpicture}
  \end{center}
    \caption{Darcy example with $d=2$ (left) and $d=10$ (right). 
    Comparison of an MCMC method and the low-rank surrogate for computing the mean error $(\err_\mu)$ with respect to the number of calls to the solution of the forward problem. 
    The reference mean is computed with $10^6$ MCMC samples.
    Additionally the KL divergence is shown, which is computed using empirical integration.
    }
    \label{fig:poisson d}
  \end{figure*}
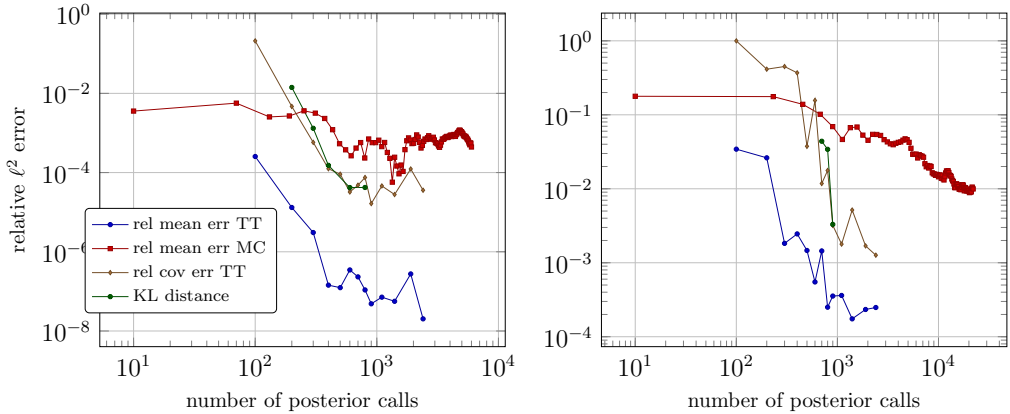
  
  The improvement of the mean and covariance estimate by the low-rank approach can already be observed for a low sample number.
  We note that the Monte Carlo estimate did not allow for an adequate computation of the empirical covariance, which therefore is left out of the comparison.
  
  \section{Conclusion} 
We developed a novel approach to approximate probability densities with high accuracy, combining the notion of \textit{transport maps} and \textit{low-rank functional representations} of auxiliary (perturbed) reference densities.
Based on a suiteable class of transformations, an approximation with respect to a finite tensorized basis can be carried out in extended hierachical tensor formats.
This yields a compressed representation for an efficient computation of statistical quantities (\eg moments or marginals) of interest in a sampling free manner.
In this work the multivariate polar transformation is used as a particular rank 1 stable transformation.
The method requires point evaluations of the perturbed reference density (up to a multiplicative constant).
The approach can hence be applied to not normalized posterior densities in the context of Bayesian inversion. 

We presented the application of the method to an inverse problem with a log-normal Darcy forward model.
A comparison with classical MCMC illustrates the superior convergence in terms of the moment accuracy relative to the number of posterior evaluations.
Future research will be concerned with
\begin{itemize}
 \item application: usage of the approximated densities for subsequent computations \eg with SGFEM, 
 \item analysis: Given a function $\tilde{f}_0$ it has to be examined which rank 1 stable transformations $\Phi$ lead to a low-rank function $\tilde{f}_0\circ \Phi$.
\end{itemize}

\section*{Acknowledgments}
The authors would like to thank Reinhold Schneider for fruitful discussions.

\bibliographystyle{plain}
\bibliography{references}

\begin{thebibliography}{10}

\bibitem{babuvska2010stochastic}
Ivo Babu{\v{s}}ka, Fabio Nobile, and Ra{\'u}l Tempone.
\newblock A stochastic collocation method for elliptic partial differential
  equations with random input data.
\newblock {\em SIAM review}, 52(2):317--355, 2010.

\bibitem{bachmayr2017parametric}
Markus Bachmayr, Albert Cohen, and Wolfgang Dahmen.
\newblock Parametric pdes: sparse or low-rank approximations?
\newblock {\em IMA Journal of Numerical Analysis}, 38(4):1661--1708, 2017.

\bibitem{bachmayr2016tensor}
Markus Bachmayr, Reinhold Schneider, and Andr{\'e} Uschmajew.
\newblock Tensor networks and hierarchical tensors for the solution of
  high-dimensional partial differential equations.
\newblock {\em Foundations of Computational Mathematics}, 16(6):1423--1472,
  2016.

\bibitem{ballani2013black}
Jonas Ballani, Lars Grasedyck, and Melanie Kluge.
\newblock Black box approximation of tensors in hierarchical tucker format.
\newblock {\em Linear algebra and its applications}, 438(2):639--657, 2013.

\bibitem{TM}
Ricardo~M. Baptista, Daniele Bigoni, Rebecca Morrison, and Alessio Spantini.
\newblock {{\tt TransportMaps}}, \verb#http://transportmaps.mit.edu/docs/#, MIT
  Uncertainty Quantification Group , 2015-2018.

\bibitem{bigoni2019greedy}
Daniele Bigoni, Olivier Zahm, Alessio Spantini, and Youssef Marzouk.
\newblock Greedy inference with layers of lazy maps.
\newblock {\em arXiv preprint arXiv:1906.00031}, 2019.

\bibitem{chen2016sparse}
Peng Chen and Christoph Schwab.
\newblock Sparse-grid, reduced-basis bayesian inversion: Nonaffine-parametric
  nonlinear equations.
\newblock {\em Journal of Computational Physics}, 316:470--503, 2016.

\bibitem{optimalweighted17}
Albert Cohen and Giovanni Migliorati.
\newblock Optimal weighted least-squares methods.
\newblock {\em The SMAI journal of computational mathematics}, 3:181--203,
  2017.

\bibitem{da2013lebesgue}
Gaspare Da~Fies and Marco Vianello.
\newblock On the lebesgue constant of subperiodic trigonometric interpolation.
\newblock {\em Journal of Approximation Theory}, 167:59--64, 2013.

\bibitem{dashti2016bayesian}
Masoumeh Dashti and Andrew~M Stuart.
\newblock The bayesian approach to inverse problems.
\newblock {\em Handbook of uncertainty quantification}, pages 1--118, 2016.

\bibitem{detommaso2018stein}
Gianluca Detommaso, Tiangang Cui, Youssef Marzouk, Alessio Spantini, and Robert
  Scheichl.
\newblock A stein variational newton method.
\newblock In {\em Advances in Neural Information Processing Systems}, pages
  9169--9179, 2018.

\bibitem{detommaso2019hint}
Gianluca Detommaso, Jakob Kruse, Lynton Ardizzone, Carsten Rother, Ullrich
  K{\"o}the, and Robert Scheichl.
\newblock Hint: Hierarchical invertible neural transport for general and
  sequential bayesian inference.
\newblock {\em arXiv preprint arXiv:1905.10687}, 2019.

\bibitem{dodwell2019multilevel}
TJ~Dodwell, C~Ketelsen, R~Scheichl, and AL~Teckentrup.
\newblock Multilevel markov chain monte carlo.
\newblock {\em Siam Review}, 61(3):509--545, 2019.

\bibitem{dolgov2018approximation}
Sergey Dolgov, Karim Anaya-Izquierdo, Colin Fox, and Robert Scheichl.
\newblock Approximation and sampling of multivariate probability distributions
  in the tensor train decomposition.
\newblock {\em arXiv preprint arXiv:1810.01212}, 2018.

\bibitem{EGSZ1}
Martin Eigel, Claude~Jeffrey Gittelson, Christoph Schwab, and Elmar Zander.
\newblock Adaptive stochastic galerkin fem.
\newblock {\em Computer Methods in Applied Mechanics and Engineering},
  270:247--269, 2014.

\bibitem{alea}
Martin Eigel, Robert Gruhlke, Manuel Marschall, and Elmar Zander.
\newblock {{\tt alea} - A Python Framework for Spectral Methods and Low-Rank
  Approximations in Uncertainty Quantification},
  \verb#https://bitbucket.org/aleadev/alea#.

\bibitem{EMPS18}
Martin Eigel, Manuel Marschall, Max Pfeffer, and Reinhold Schneider.
\newblock Adaptive stochastic galerkin fem for lognormal coefficients in
  hierarchical tensor representations.
\newblock {\em arXiv preprint arXiv:1811.00319}, 2018.

\bibitem{EMS18}
Martin Eigel, Manuel Marschall, and Reinhold Schneider.
\newblock Sampling-free bayesian inversion with adaptive hierarchical tensor
  representations.
\newblock {\em Inverse Problems}, 34(3):035010, 2018.

\bibitem{eigel2019non}
Martin Eigel, Johannes Neumann, Reinhold Schneider, and Sebastian Wolf.
\newblock Non-intrusive tensor reconstruction for high-dimensional random pdes.
\newblock {\em Computational Methods in Applied Mathematics}, 19(1):39--53,
  2019.

\bibitem{EPS17}
Martin Eigel, Max Pfeffer, and Reinhold Schneider.
\newblock Adaptive stochastic galerkin fem with hierarchical tensor
  representations.
\newblock {\em Numerische Mathematik}, 136(3):765--803, 2017.

\bibitem{ESTW19}
Martin Eigel, Reinhold Schneider, Philipp Trunschke, and Sebastian Wolf.
\newblock Variational monte carlo---bridging concepts of machine learning and
  high-dimensional partial differential equations.
\newblock {\em Advances in Computational Mathematics}, 10 2019.

\bibitem{el2012bayesian}
Tarek~A El~Moselhy and Youssef~M Marzouk.
\newblock Bayesian inference with optimal maps.
\newblock {\em Journal of Computational Physics}, 231(23):7815--7850, 2012.

\bibitem{ernst2019expansions}
Oliver~G Ernst, Bj{\"o}rn Sprungk, and Lorenzo Tamellini.
\newblock On expansions and nodes for sparse grid collocation of lognormal
  elliptic pdes.
\newblock {\em arXiv preprint arXiv:1906.01252}, 2019.

\bibitem{espig2009black}
Mike Espig, Lars Grasedyck, and Wolfgang Hackbusch.
\newblock Black box low tensor-rank approximation using fiber-crosses.
\newblock {\em Constructive approximation}, 30(3):557, 2009.

\bibitem{foo2010multi}
Jasmine Foo and George~Em Karniadakis.
\newblock Multi-element probabilistic collocation method in high dimensions.
\newblock {\em Journal of Computational Physics}, 229(5):1536--1557, 2010.

\bibitem{garcke2012sparse}
Jochen Garcke and Michael Griebel.
\newblock {\em Sparse grids and applications}, volume~88.
\newblock Springer Science \& Business Media, 2012.

\bibitem{gilks1995markov}
Walter~R Gilks, Sylvia Richardson, and David Spiegelhalter.
\newblock {\em Markov chain Monte Carlo in practice}.
\newblock Chapman and Hall/CRC, 1995.

\bibitem{gorodetsky2015function}
Alex~A Gorodetsky, Sertac Karaman, and Youssef~M Marzouk.
\newblock Function-train: A continuous analogue of the tensor-train
  decomposition.
\newblock {\em arXiv preprint arXiv:1510.09088}, 2015.

\bibitem{griebel2013construction}
Michael Griebel and Helmut Harbrecht.
\newblock On the construction of sparse tensor product spaces.
\newblock {\em Mathematics of computation}, 82(282):975--994, 2013.

\bibitem{hackbusch2012tensor}
Wolfgang Hackbusch.
\newblock {\em Tensor spaces and numerical tensor calculus}, volume~42.
\newblock Springer Science \& Business Media, 2012.

\bibitem{holtz2012alternating}
Sebastian Holtz, Thorsten Rohwedder, and Reinhold Schneider.
\newblock The alternating linear scheme for tensor optimization in the tensor
  train format.
\newblock {\em SIAM Journal on Scientific Computing}, 34(2):A683--A713, 2012.

\bibitem{xerus}
Benjamin Huber and Sebastian Wolf.
\newblock Xerus - a general purpose tensor library.
\newblock \url{https://libxerus.org/}, 2014--2017.

\bibitem{kaipio2006statistical}
Jari Kaipio and Erkki Somersalo.
\newblock {\em Statistical and computational inverse problems}, volume 160.
\newblock Springer Science \& Business Media, 2006.

\bibitem{li2014adaptive}
Jinglai Li and Youssef~M Marzouk.
\newblock Adaptive construction of surrogates for the bayesian solution of
  inverse problems.
\newblock {\em SIAM Journal on Scientific Computing}, 36(3):A1163--A1186, 2014.

\bibitem{liu2016stein}
Qiang Liu and Dilin Wang.
\newblock Stein variational gradient descent: A general purpose bayesian
  inference algorithm.
\newblock In {\em Advances in neural information processing systems}, pages
  2378--2386, 2016.

\bibitem{marzouk2016introduction}
Youssef Marzouk, Tarek Moselhy, Matthew Parno, and Alessio Spantini.
\newblock An introduction to sampling via measure transport.
\newblock {\em arXiv preprint arXiv:1602.05023}, 2016.

\bibitem{mead1973convergence}
KO~Mead and LM~Delves.
\newblock On the convergence rate of generalized fourier expansions.
\newblock {\em IMA Journal of Applied Mathematics}, 12(3):247--259, 1973.

\bibitem{neal2001annealed}
Radford~M Neal.
\newblock Annealed importance sampling.
\newblock {\em Statistics and computing}, 11(2):125--139, 2001.

\bibitem{nobile2008sparse}
Fabio Nobile, Ra{\'u}l Tempone, and Clayton~G Webster.
\newblock A sparse grid stochastic collocation method for partial differential
  equations with random input data.
\newblock {\em SIAM Journal on Numerical Analysis}, 46(5):2309--2345, 2008.

\bibitem{oseledets2010tt}
Ivan Oseledets and Eugene Tyrtyshnikov.
\newblock {TT}-cross approximation for multidimensional arrays.
\newblock {\em Linear Algebra and its Applications}, 432(1):70--88, 2010.

\bibitem{oseledets2011tensor}
Ivan~V Oseledets.
\newblock Tensor-train decomposition.
\newblock {\em SIAM Journal on Scientific Computing}, 33(5):2295--2317, 2011.

\bibitem{papamakarios2019normalizing}
George Papamakarios, Eric Nalisnick, Danilo~Jimenez Rezende, Shakir Mohamed,
  and Balaji Lakshminarayanan.
\newblock Normalizing flows for probabilistic modeling and inference.
\newblock {\em arXiv preprint arXiv:1912.02762}, 2019.

\bibitem{parno2016multiscale}
Matthew Parno, Tarek Moselhy, and Youssef Marzouk.
\newblock A multiscale strategy for bayesian inference using transport maps.
\newblock {\em SIAM/ASA Journal on Uncertainty Quantification},
  4(1):1160--1190, 2016.

\bibitem{parno2018transport}
Matthew~D Parno and Youssef~M Marzouk.
\newblock Transport map accelerated markov chain monte carlo.
\newblock {\em SIAM/ASA Journal on Uncertainty Quantification}, 6(2):645--682,
  2018.

\bibitem{paszke2017automatic}
Adam Paszke, Sam Gross, Soumith Chintala, Gregory Chanan, Edward Yang, Zachary
  DeVito, Zeming Lin, Alban Desmaison, Luca Antiga, and Adam Lerer.
\newblock Automatic differentiation in pytorch.
\newblock In {\em NIPS-W}, 2017.

\bibitem{rezende2015variational}
Danilo~Jimenez Rezende and Shakir Mohamed.
\newblock Variational inference with normalizing flows.
\newblock {\em arXiv preprint arXiv:1505.05770}, 2015.

\bibitem{rohrbach2020rank}
Paul~B Rohrbach, Sergey Dolgov, Lars Grasedyck, and Robert Scheichl.
\newblock Rank bounds for approximating gaussian densities in the tensor-train
  format.
\newblock {\em arXiv preprint arXiv:2001.08187}, 2020.

\bibitem{rudolf2017metropolis}
Daniel Rudolf and Bj{\"o}rn Sprungk.
\newblock Metropolis-hastings importance sampling estimator.
\newblock {\em PAMM}, 17(1):731--734, 2017.

\bibitem{santambrogio2015optimal}
Filippo Santambrogio.
\newblock Optimal transport for applied mathematicians.
\newblock {\em Birk{\"a}user, NY}, 55:58--63, 2015.

\bibitem{schillings2016scaling}
Claudia Schillings and Christoph Schwab.
\newblock Scaling limits in computational bayesian inversion.
\newblock {\em ESAIM: Mathematical Modelling and Numerical Analysis},
  50(6):1825--1856, 2016.

\bibitem{schillings2019convergence}
Claudia Schillings, Bj{\"o}rn Sprungk, and Philipp Wacker.
\newblock On the convergence of the laplace approximation and
  noise-level-robustness of laplace-based monte carlo methods for bayesian
  inverse problems.
\newblock {\em arXiv preprint arXiv:1901.03958}, 2019.

\bibitem{schneider2014approximation}
Reinhold Schneider and Andr{\'e} Uschmajew.
\newblock Approximation rates for the hierarchical tensor format in periodic
  sobolev spaces.
\newblock {\em Journal of Complexity}, 30(2):56--71, 2014.

\bibitem{schwab2011sparse}
Christoph Schwab and Claude~Jeffrey Gittelson.
\newblock Sparse tensor discretizations of high-dimensional parametric and
  stochastic pdes.
\newblock {\em Acta Numerica}, 20:291--467, 2011.

\bibitem{stuart2010inverse}
Andrew~M Stuart.
\newblock Inverse problems: a bayesian perspective.
\newblock {\em Acta numerica}, 19:451--559, 2010.

\bibitem{tran2019discrete}
Dustin Tran, Keyon Vafa, Kumar~Krishna Agrawal, Laurent Dinh, and Ben Poole.
\newblock Discrete flows: Invertible generative models of discrete data.
\newblock {\em arXiv preprint arXiv:1905.10347}, 2019.

\bibitem{villani2008optimal}
C{\'e}dric Villani.
\newblock {\em Optimal transport: old and new}, volume 338.
\newblock Springer Science \& Business Media, 2008.

\bibitem{weare2007efficient}
{J}onathan {W}eare.
\newblock Efficient {M}onte {C}arlo sampling by parallel marginalization.
\newblock {\em Proceedings of the {N}ational {A}cademy of {S}ciences},
  104(31):12657--12662, 2007.

\end{thebibliography}

\end{document}